\documentclass[12pt]{article}
\usepackage[english]{babel}
\usepackage{amsfonts,amsmath,amssymb,amsthm,graphicx,afterpage,url}
\input{xy}
\xyoption{all}
\usepackage{overpic}
\usepackage{pgf,tikz}
\usetikzlibrary{arrows}
\usepackage{epsfig,epstopdf}

\newcommand{\algor}[1]{}
\newcommand{\invadraw}[1]{#1}
\newcommand{\jonly}[1]{}

\graphicspath
{{morefig/}{figures/}{figures00/}{figures02/}{figures05/}{figlms/}{obstruct-final/}{figumisha/}{figurestas/}}

\makeatletter

\@addtoreset{figure}{section}
\makeatother

\def\con{\mbox{con}}

\def\Sq{\mathop{\fam0 Sq}}
\def\diag{\mathop{\fam0 diag}}

\def\con{\mathop{\fam0 Con}}
\def\id{\mathop{\fam0 id}}

\def\sgn{\mathop{\fam0 sgn}}
\def\gsgn{\mathop{\fam0 gsgn}}
\def\csgn{\mathop{\fam0 csgn}}

\def\t{\widetilde}
\def\R{{\mathbb R}} \def\Z{{\mathbb Z}}

\long\def\comment#1\endcomment{}

\theoremstyle{theorem}
\newtheorem{theorem}{Theorem}[subsection]
    \newtheorem{lemma}[theorem]{Lemma}
    
    \newtheorem{proposition}[theorem]{Proposition}
    \newtheorem{conjecture}[theorem]{Conjecture}

\theoremstyle{definition}
\newtheorem{remark}[theorem]{Remark}
\newtheorem{example}[theorem]{Example}
\newtheorem{pr}[theorem]{Assertion}
\newtheorem{problem}[theorem]{Problem}

\begin{document}


\title{Invariants of graph drawings in the plane}

\author{A. Skopenkov
\footnote{Moscow Institute of Physics and Technology, Independent University of Moscow.
Supported in part by the Simons-IUM fellowship and Russian Foundation for Basic Research Grant No. 19-01-00169.
Email: \texttt{skopenko@mccme.ru}.
\texttt{https://users.mccme.ru/skopenko/}.}
}

\date{}

\maketitle

\begin{abstract}
We present a simplified exposition of some classical and modern results on graph drawings in the plane.
These results are chosen so that they illustrate some spectacular recent higher-dimensional results on the border of geometry, combinatorics and topology.
We define a $\Z_2$-valued {\it self-intersection invariant} (i.e. the van Kampen number) and its generalizations.
We present elementary formulations and arguments accessible to mathematicians not specialized in any of the areas discussed.
So most part of this survey could be studied before textbooks on algebraic topology, as an introduction to starting ideas of algebraic topology motivated by algorithmic, combinatorial and geometric problems.
\end{abstract}


\tableofcontents

\subsection*{Introduction}\label{s:intrgr}
\addcontentsline{toc}{subsection}{Introduction}

{\bf Why this survey could be interesting.}
In this survey we present a simplified exposition of some classical and modern results on graph drawings in the plane (\S\ref{0-gra}, \S\ref{s:mucoge}).
These results are chosen so that they illustrate some spectacular recent higher-dimensional results on the border  of geometry, combinatorics and topology (\S\ref{s:high}).

We exhibit a connection between non-planarity of the complete graph $K_5$ on 5 vertices and
results on intersections in the plane of {\it algebraic interiors} of curves (namely, the Topological Radon-Tverberg Theorems in the plane \ref{ratv-totv}, \ref{ratv-tvpl}).
Recent resolution of the Topological Tverberg Conjecture \ref{c:tt} on multiple intersections for maps from simplex to Euclidean space used a higher-dimensional $r$-fold generalization of this connection (i.e. a connection between the $r$-fold van Kampen-Flores Conjecture \ref{c:vkfg} and Conjecture \ref{c:tt}).

Recall that invariants of knots were initially defined using {\it presentations of the fundamental group} at the beginning of the 20th century and, even in a less elementary way, at the end of the 20th century.
An elementary description of knot invariants via plane diagrams (initiated in J. Conway's work of the second half of the 20th century) increased interest in knot theory and made that
part of topology a part of graph theory as well.

Analogously, we present elementary formulations and arguments that do not
involve configuration spaces and cohomological obstructions.
Nevertheless, the main contents of this survey is an {\it introduction to starting ideas of
algebraic topology} (more precisely, to configuration spaces and
cohomological obstructions) {\it motivated by algorithmic, combinatorial and geometric problems.}
We believe that describing simple applications of topological methods in elementary language makes
these methods more accessible (although this is called `detopologization' in \cite[\S1]{MTW12}).
Such an introduction \jonly{is independent of} could be studied before textbooks on algebraic topology
(if a reader is ready to accept without proof some results from \S\ref{s:rainbow}).
For textbooks written in the spirit of this article see e.g. \cite{Sk20, Sk}.


More precisely, it is fruitful to invent or to interpret homotopy-theoretical arguments in terms of invariants defined via intersections or preimages.\footnote{Examples are definition of the mapping degree \cite[\S2.4]{Ma03}, \cite[\S8]{Sk20} and definition of the Hopf invariant via linking, i.e. via intersection \cite[\S8]{Sk20}.
Importantly, `secondary' not only `primary' invariants allow interpretations in terms of {\it framed} intersections; for a recent application see \cite{Sk17}.}
In this survey we describe in terms of double and multiple intersection numbers those arguments that are often exposed in a less elementary language of homotopy theory.

No knowledge of algebraic topology is required here.
Important ideas are introduced in non-technical particular cases and then generalized.\footnote{The `minimal generality' principle (to introduce important ideas in non-technical particular cases) was put forward by classical figures in mathematics and mathematical exposition, in particular by V. Arnold.
Cf. `detopologization' tradition described in \cite[Historical notes in \S1]{MTW12}.}
So this survey is accessible to mathematicians not specialized in the area.


{\bf Contents of this survey.}
Both \S\ref{0-gra} and \S\ref{s:mucoge} bring the reader to the frontline of research.

In \S\ref{0-gra} we present a polynomial algorithm for recognizing graph planarity (\S\ref{0vkam2}),
together with all the necessary definitions, some motivations  and preliminary results (\S\S\ref{0hygr}--\ref{0grapl}).
This algorithm, the corresponding planarity criterion (Proposition \ref{t:pl-vk}) and a simple proof of the non-planarity of $K_5$ (\S\ref{0grapl}) are interesting because they can be generalized to higher dimensions and higher multiplicity of intersections (Theorems \ref{t:lvkf}, \ref{t:tvkf}, \ref{t:rec}, \ref{t:nphh}, \ref{t:ozmawaco}  and  \ref{t:algal}, see also  Conjectures \ref{c:lvkfg} and \ref{c:vkfg}).


In \S\ref{s:mucoge} we introduce in a simplified way results on multiple intersections in the plane of {\it algebraic interiors} of curves  (namely, the topological Radon-Tverberg theorems \ref{ratv-totv},  \ref{ratv-tvpl}, and  the topological Tverberg conjecture in the plane \ref{ratv-tvplc}).
A generalization of the ideas from \S\ref{0grapl}, \S\ref{0-ratvtopl}, \cite[Lemmas 6 and 7]{ST17}  could give a simple proof of the topological Tverberg theorem, and of its `quantitative' version, at least for primes (\S\ref{s:triple}).
This is interesting in particular because the topological Tverberg conjecture in the plane \ref{ratv-tvplc} is still open.
We also give a simplified formulation of the \"Ozaydin theorem in the plane \ref{vankamz-oz} on cohomological obstructions for multiple intersections of algebraic interiors of curves.
This formulation can perhaps be applied to obtain a simple proof.


In \S\ref{s:high} we indicate how elementary results of \S\ref{0-gra} and \S\ref{s:mucoge} illustrate some spectacular recent higher-dimensional results.
Detailed description of those recent results is outside purposes of this survey.
In \S\ref{s:hirtvkf} we state classical and modern results and conjectures on complete hypergraphs (since the results only concern {\it complete} hypergraphs, we present simplified statements not involving hypergraphs).
These results generalize non-planarity of $K_5$ (Proposition \ref{0-ra2}.a and Theorem \ref{grapl-nonalm})
and the results on intersections of algebraic interiors of curves (linear and topological Radon and Tverberg theorems in the plane \ref{0-radpl}, \ref{ratv-tv1}, \ref{ratv-totv}, \ref{ratv-tvpl}).
In \S\ref{s:hialg} we state modern algorithmic results on realizability of arbitrary hypergraphs; they generalize Proposition \ref{grapl-ea}.b.
In \S\ref{s:hiae} we do the same for {\it almost realizability}.
This notion is defined there but implicitly appeared in \S\ref{0grapl}, \S\ref{s:mucoge}.
We introduce \"Ozaydin Theorem \ref{t:ozmawa}, which is a higher-dimensional version of the above-mentioned
\"Ozaydin Theorem in the plane \ref{vankamz-oz}, and which is an important ingredient in recent resolution of the topological Tverberg conjecture \ref{c:tt}.

The main notion of this survey linking together \S\ref{0-gra} and \S\ref{s:mucoge} is a $\Z_2$-valued `self-intersection' invariant  (i.e. the van Kampen and the Radon numbers defined in \S\ref{0grapl}, \S\ref{0-ratvtopl}).
Its generalizations to $\Z_r$-valued invariants and to cohomological obstructions are defined and used to obtain elementary formulations and proofs of \S\ref{0-gra} and \S\ref{s:mucoge} mentioned above.
For applications of other generalizations see \cite[\S4]{Sk16}, \cite{Sk16', ST17}.
For invariants of plane curves and caustics see \cite{Ar95} and the references therein.

\begin{remark}[generalizations in five different directions]\label{r:five}
The main results exposed in this survey can be obtained from the easiest of them (Linear van Kampen-Flores and Radon Theorems for the plane \ref{0-ra2}.a, \ref{0-radpl}) by generalizations in five different directions.
Thus the results can naturally be numbered by a vector of length five.

First, a result can give intersection of simplices of some dimensions, or of the same dimension.
This relates \S\ref{0-gra} to \S\ref{s:mucoge}.

Second, a `qualitative' result on the existence of intersection can be generalized to a `quantitative' result on  the algebraic number of intersections.
(This relates Proposition \ref{0-ra2}.a to Proposition \ref{0-ra2}.b,
Theorem \ref{ratv-tvpl} to Problem \ref{p:boundary}, etc.)

Third, a linear result can be generalized to a topological result, which is here equivalent to a piecewise linear result.
(This relates Propositions \ref{0-ra2}.ab to Theorem \ref{grapl-nonalm} and Lemma \ref{11-vankam}, etc.)

Fourth, a result on double intersection can be generalized to multiple intersections.
(This relates Proposition \ref{0-ra2}.a and Theorem \ref{grapl-nonalm} to Theorems  \ref{ratv-tv1} and \ref{ratv-tvpl}, etc; note that the $r$-tuple intersection version might not hold for $r$ not a power of a prime.)

Fifth, a result in the plane can be generalized to higher dimensions.
This relates \S\ref{0-gra} and \S\ref{s:mucoge} to \S\ref{s:high}.
\end{remark}


{\bf Structure of this survey.}
Subsections of this survey (except appendices) can be read independently of each other, and so in any order.
In one subsection we indicate relations to other subsections, but these indications can be ignored.
If in one subsection we use a definition or a result from the other, then we only use a specific shortly stated  definition or result.
However, we recommend to read subsections in any order consistent with the following diagram.
$$\xymatrix{
\ref{0hygr} \ar[r] &\ref{s:grapl} \ar[r] &\ref{0grapl} \ar[r] \ar[dr] \ar@(ur,ul)[rr]
&\ref{0vkam2} \ar[drr] &\ref{s:hirtvkf} \ar[r] &\ref{s:hialg} \ar[r] &\ref{s:hiae}\\
&\ref{0thint} \ar@(ur,ul)[rr] \ar[ur] &\ref{s:ratvpl} \ar[r] & \ref{0-ratvtopl} \ar[r] &\ref{s:ttp} \ar[r] \ar[u] &\ref{s:tvkam} \ar[u]
}$$
Main statements are called theorems, important statements are lemmas or propositions, less important statements which are not referred to outside a subsection are assertions.
Less important or better known material is moved to appendices.




{\bf Historical notes}.
All the results of this survey are known.

For history, more motivation,
more proofs, related problems and generalizations see surveys \cite{BBZ, Zi11, Sk16, BZ, Sh18} (to \S\ref{s:mucoge} and \S\ref{s:hirtvkf}) and \cite{Sk06, Sk14}, \cite[\S1]{MTW}, \cite[\S5 `Realizability of complexes']{Sk} (to \S\ref{0-gra} and \S\ref{s:hialg}).
Discussion of those related problems and generalizations is outside purposes of this survey.

I do not give original references to trivial or standard results (e.g. to Propositions \ref{0-ra2}
and \ref{ratv-totv2}), as well as to classical results
for which original references are given in the above surveys (e.g. to F\'ary or Radon theorems \ref{grapl-fary}, \ref{0-radpl}, \ref{t:lr}).
I do give original references to modern results from \cite{HT74, BB} on.
I also refer to some proofs which are not original but which could be useful to the reader (for example, in connection with this survey).

Exposition of the polynomial algorithm for recognizing graph planarity (\S\ref{0vkam2}) is new.
First, we give an elementary statement of the corresponding planarity criterion (Proposition \ref{t:pl-vk}).
Second, we do not require knowledge of cohomology theory but show how some notions of that theory naturally appear in studies planarity of graphs.
Cf. \cite{Fo04}, \cite[Appendix D]{MTW}, \cite[\S1.4.2]{Sc13}.

Elementary formulation of the topological Radon theorem (\S\ref{0-ratvtopl}) in the spirit of \cite{Sc04, SZ} is
presumably folklore.
The proof follows the idea of L. Lovasz and A. Schrijver \cite{LS98}.
Elementary formulation of the topological Tverberg theorem and conjecture in the plane (\S\ref{s:tvtopl})
is due to T. Sch\"oneborn and G. Ziegler  \cite{Sc04, SZ}.
An idea of an elementary proof of that result (\S\ref{s:triple}) and a simplified formulation of M. \"Ozaydin's results (\S\ref{s:tvkam}) are apparently new.

The paper \cite{ERS} was used in preparation of the first version of this paper; most part of the first version of \S\ref{s:mucoge} is written jointly with A. Ryabichev; some proofs from \S\ref{s:appgr} were written by A. Ryabichev and T. Zaitsev.
I am grateful to P. Blagojevi{\'c}, I. Bogdanov, G. Chelnokov, A. Enne, R. Fulek, R. Karasev, Yu. Makarychev, A. Ryabichev, M. Schaefer, G. Sokolov, M. Tancer, T. Zaitsev, R. \v Zivaljevi\'c and anonymous referees for useful discussions.


{\bf Conventions.}
Unless the opposite is
indicated, by {\it $k$ points in the plane} we mean a $k$-element subset of the plane; so these $k$ points are assumed to be pairwise distinct.
We often denote points by numbers not by letters with subscript numbers.
Denote $[n]:=\{1,2,\ldots,n\}$.

\section{Planarity of graphs}\label{0-gra}

A (finite) {\bf graph} $(V,E)$ is a finite set $V$ together with a collection $E\subset {V\choose 2}$
of two-element subsets of $V$ (i.e. of non-ordered pairs of elements).\footnote{The common term for this notion is {\it a graph without loops and multiple edges} or {\it a simple graph}.}
The elements of this finite set $V$ are called {\it vertices}.
Unless otherwise indicated, we assume that $V=[|V|]$.
The pairs of vertices from $E$ are called {\it edges}.
The edge joining vertices $i$ and $j$ is denoted by $ij$ (not by $(i,j)$ to avoid confusion with ordered pairs).

\begin{figure}[h]\centering
\includegraphics{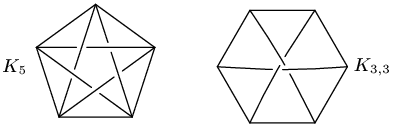}\qquad\qquad \includegraphics[scale=0.5]{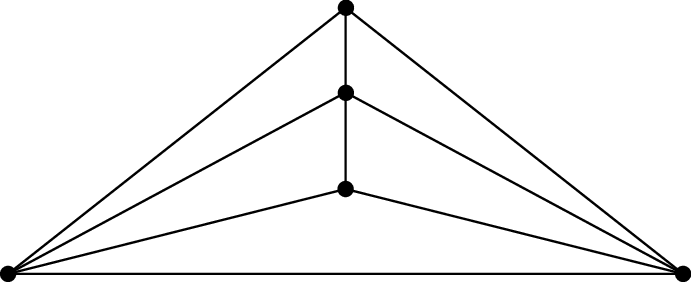}
\caption{(Left) Nonplanar graphs $K_5$ and $K_{3,3}$.
\newline
(Right) A planar drawing of $K_5$ without one of the edges.}\label{k5}
\end{figure}

A {\it complete graph} $K_n$ on $n$ vertices is a graph in which every pair of vertices is connected by an edge, i.e.  $E={V\choose 2}$.
A {\it complete bipartite graph} $K_{m,n}$  is a graph whose vertices can be partitioned into two subsets of $m$ elements and of $n$ elements, so that

$\bullet$ every two vertices from different subsets are joined by an edge, and

$\bullet$ every edge connects vertices from different subsets.


In \S\ref{0hygr} and \S\ref{s:grapl} we present two formalizations
of realizability of graphs in the plane: the linear realizability and the planarity (i.e. piecewise linear realizability).
The formalizations turn out to be equivalent by F\'ary Theorem \ref{grapl-fary}; their higher-dimensional generalizations (\S\ref{s:hialg}) are not equivalent, see \cite{vK41}, \cite[\S2]{MTW}.
Both formalizations are important.
These formalizations are presented independently of each other, so \S\ref{0hygr} is essentially not used below (except for Proposition \ref{0-ra2}.b making the proof of Lemma \ref{11-vankam} easier, and footnote \ref{f:alg}, which are trivial and not important).
However, before more complicated study of planarity it could be helpful to study linear realizability.
The tradition of studying both linear and piecewise linear problems is also important for \S\ref{s:mucoge}, see Remark \ref{r:five}.

\subsection{Linear realizations of graphs}\label{0hygr}

\begin{proposition}\label{0-ra2}\footnote{These are `linear' versions of the nonplanarity of the graphs $K_5$ and $K_{3,3}$.
But they can be proved easier (because the Parity Lemma \ref{0-even}.b and \cite[Intersection Lemma 1.4.4]{Sk20} are not required for the proof). }
(a) (cf. Theorems \ref{grapl-nonalm} and \ref{0-radpl})
From any $5$ points in the plane one can choose two disjoint pairs such that the segment with the ends at the first pair intersects the segment with the ends at the second pair.

(b) (cf. Proposition \ref{ratv-vk2l} and Lemma \ref{11-vankam})
If no $3$ of $5$ points in the plane belong to a line, then the number of intersection points of interiors of segments joining the $5$ points is odd.


(c) (cf. Remark \ref{grapl-k33}.d) Two triples of points in the plane are given.
Then there exist two intersecting segments without common vertices and such that each segment joins points from distinct triples.
\end{proposition}

Proposition \ref{0-ra2} is easily proved by analyzing the {\it convex hull} of the points  (see definition in \S\ref{s:ratvpl}).
See another proof in \jonly{\cite[\S1.6]{Sk18}.} \S\ref{s:appgr}.
For part (c) the analysis is lengthy, so using methods from the proof of Lemmas \ref{11-vankam} and \ref{star} might be preferable.

\begin{theorem}[General Position; see proof in \S\ref{s:appgr}]\label{0-gp3}
For any $n$ there exist $n$ points in $3$-space such that no segment joining the points intersects the interior of any other such segment.
\end{theorem}

\begin{proposition}\label{1-k5-1}\footnote{See proof in \jonly{\cite[\S1.6]{Sk18}.} \S\ref{s:appgr}.
Propositions \ref{1-k5-1} and
\jonly{\cite[1.6.1]{Sk18}} \ref{grapl-ram}
are not formally used in this paper.
However, they illustrate by 2-dimensional examples how boolean functions appear in the study of embeddings.
This is one of the ideas behind recent higher-dimensional
$NP$-hardness Theorem \ref{t:nphh}.}
Suppose that no 3 of 5 points $1,2,3,4,5$ in the plane belong to a line.
If the segments

(a) $jk$, $1\le j<k\le 5$, $(j,k)\ne(1,2)$, have disjoint interiors then the points $1$ and $2$ lie on different  sides of the triangle $345$, cf. figure \ref{k5}, right;

(b) $jk$, $1\le j<k\le 5$, $(j,k)\not\in\{(1,2),(1,3)\}$, have disjoint interiors then

EITHER the points $1$ and $2$ lie on different sides of the triangle $345$,

OR the points $1$ and $3$ lie on different sides of the triangle $245$.

(c) $jk$, $1\le j<k\le 5$, $(j,k)\not\in\{(1,2),(1,3),(1,4)\}$, have disjoint interiors then

EITHER the points $1$ and $2$ lie on different sides of the triangle $345$,

OR the points $1$ and $3$ lie on different sides of the triangle $245$,

OR the points $1$ and $4$ lie on different sides of the triangle $235$.

\end{proposition}

Informally speaking, a graph is {\it linearly realizable} in the plane if the graph has a planar drawing without self-intersection and such that every edge is drawn as a line segment.
Formally, a graph $(V,E)$ is called {\bf linearly realizable} in the plane if there exists $|V|$ points in the plane corresponding to the vertices so that no segment joining a pair (of points) from $E$
intersects the interior of any other such segment.\footnote{We do not require that `no isolated vertex lies on any of the segments'
because this property can always be achieved.}


The following results are classical:

$\bullet$ $K_4$ and $K_5$ without one of the edges are linearly realizable in the plane (figure \ref{k5}, right).

$\bullet$ neither $K_5$ nor $K_{3,3}$ is linearly realizable in the plane (Proposition \ref{0-ra2}.ac);

$\bullet$ every graph is linearly realizable in $3$-space
(General Position Theorem \ref{0-gp3}; linear
realizability in $3$-space is defined analogously to the plane).

A criterion for linear realizability of graphs in the plane follows from the F\'ary Theorem \ref{grapl-fary} below and any planarity criterion (e.g. Kuratowski Theorem \ref{grapl-kur} below).


\begin{proposition}[{\cite[Chapters 1 and 6]{Ta}}; cf. \S\ref{s:hialg}; see comments in \jonly{\cite[\S1.6]{Sk18}}
\S\ref{s:appgr}] \label{1-alg} There is an algorithm for recognizing the linear realizability of graphs in the plane.\footnote{\label{f:alg} Rigorous definition of the notion of algorithm is complicated, so we do not give it here.
Intuitive understanding of algorithms is sufficient to read this text.
To be more precise, the above statement means that there is an algorithm for calculating the function from the set of all graphs to $\{0,1\}$, which maps graph to 1 if the graph is linearly realizable in the plane, and to 0 otherwise.
All other statements on algorithms in this paper can be formalized analogously.}
\end{proposition}

By the F\'ary Theorem \ref{grapl-fary} and Propositions \ref{grapl-ea}.bc polynomial and even linear algorithms exist.

\subsection{Algorithmic results on graph planarity}\label{s:grapl}


Informally speaking, a graph is planar if it can be drawn `without self-intersections' in the plane.
Formally, a graph $(V,E)$ is called {\bf planar} (or piecewise-linearly realizable in the plane) if in the plane there exist

$\bullet$ a set of $|V|$ points corresponding to the vertices, and

$\bullet$ a set of non-self-intersecting polygonal lines joining pairs (of points) from $E$

such that no of the polygonal lines intersects the interior of any other polygonal line.\footnote{Then any two of the polygonal lines either are disjoint or intersect by a common end vertex.
We do not require that `no isolated vertex lies on any of the polygonal lines'
because this property can always be achieved.
See an equivalent definition of planarity in the beginning of \S\ref{0grapl}.}

For example, the graphs $K_5$ and $K_{3,3}$ (fig.~\ref{k5}) are not planar by Theorem \ref{grapl-nonalm} and its analogue for $K_{3,3}$ \jonly{\cite[Remark 1.4.4]{Sk18}.} Remark \ref{r:grapl-k33}.

The following theorem shows that any planar graph can be drawn without self-intersections in the plane so that every edge is drawn as a segment.

\begin{theorem}[F\'ary]\label{grapl-fary} If a graph is planar (i.e. piecewise-linearly realizable in the plane), then it is linearly realizable in the plane.
\end{theorem}

For history (involving more mathematicians to whom this theorem is attributed) and proofs see 
\cite[Chapter 6]{Ta}.


\begin{proposition}\label{grapl-ea} (a) There is an algorithm for recognizing graph planarity.


(b) (cf. Theorems \ref{grapl-ea-r} and \ref{t:rec}) There is an algorithm for recognizing graph planarity, which is polynomial in the number of vertices $n$ in the graph (i.e. there are numbers $C$ and $k$ such that for each graph the number of steps in the algorithm does not exceed $Cn^k$).%
\footnote{Since for a planar graph with $n$ vertices and $e$ edges we have $e \le 3n-6$ and since there are planar graphs with  $n$ vertices and $e$ edges such that $e=3n-6$, the `complexity' in the number of edges is `the same' as the `complexity' in the number of vertices.}

(c) There is an algorithm for recognizing graph planarity, which is linear in the number of vertices $n$ in the graph (linearity is defined as polynomiality with $k=1$).
\end{proposition}

Part (a) follows from Proposition~\ref{1-alg} and the F\'ary Theorem \ref{grapl-fary}.
Part (a) can also be proved using Kuratowski Theorem \ref{grapl-kur} below (see for details \cite[Chapters 1 and 6]{Ta}) or considering {\it thickenings} \cite[\S1]{Sk20}.
However, the corresponding algorithms are slow, i.e. have more than $2^n$ steps, if the graph has $n$ vertices (`exponential complexity').
So other ways of recognizing planarity are interesting.

Part (b) is deduced from equivalence of planarity and solvability of certain system of linear equations
with coefficients in $\Z_2$ (see $(i)\Leftrightarrow(iii)$ of Proposition \ref{t:pl-vk} below).
The deduction follows because there is a polynomial in $N$ algorithm for recognizing the solvability of a system of $\le N$ linear equations with coefficients in $\Z_2$ and with $\le N$ variables (this algorithm is constructed using {\it Gauss elimination of variables algorithm}, see details in \cite{CLR, Vi02}).

Part (c) is proved in \cite{HT74}, see a short proof in \cite{BM04}.
The algorithm does not generalize to higher dimensions (as opposed to the algorithm of (b)).


\begin{figure}[h]\centering
\includegraphics[scale=0.8]{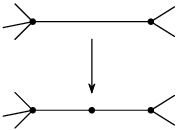}
\caption{Subdivision of edge}\label{podr}
\end{figure}

The {\it subdivision of edge} operation for a graph is shown in fig.~\ref{podr}.
Two graphs are called {\it homeomorphic} if one can be obtained from the other by subdivisions of edges and inverse operations.
This is equivalent to the existence of a graph that can be obtained from each of these graphs by subdivisions of  edges.
Some motivations for this definition are given in \cite[\S5.3]{Sk20}.

It is clear that homeomorphic graphs are either both planar or both non-planar.

A graph is planar if and only if some graph homeomorphic to it is linearly realizable in the plane.\footnote{This follows by the definition of planarity.
If the graph is planar, then every edge is presented by a polygonal line.
Define a new graph as follows: the vertices of a new graph correspond to the vertices of the polygonal line, and the edges of a new graph correspond to the edges of the polygonal line. The proof of the converse implication is analogous.}

\begin{theorem}[Kuratowski]\label{grapl-kur}
A graph is planar if and only if it has no subgraphs homeomorphic to $K_5$ or $K_{3,3}$ (fig.~\ref{k5}).
\end{theorem}

For history (involving more mathematicians to whom this theorem is attributed) and proofs see \cite{Th81}.
A particularly simple proof of this theorem by Makarychev can be found e.g. in \cite{Ma97}, \cite[\S2.9]{GDI}.


\subsection{Intersection number for polygonal lines in the plane}\label{0thint}



Before starting to read this section a reader might want to look at Assertion \ref{pla-kon} and applications from \cite[\S1.4]{Sk20}.

\jonly{Before reading this section a reader might want to look at \cite[Assertion 1.3.4]{Sk18} and applications from \cite[\S1.4]{Sk20}.
Comments and proofs are also presented in \cite[\S1.3]{Sk18}.}

Some points in the plane {\bf are in general position}, if no three of them lie in a line and no three segments joining them have a common interior point.

\begin{proposition}\label{grapl-gp} Any two polygonal lines in the plane whose vertices are in general position intersect at a finite number of points.
\end{proposition}

\begin{proof} A polygonal line is a finite union of segments.
Every two segments in general position intersect at a finite number of points.
\end{proof}

\begin{lemma}[Parity]\label{0-even} (a) If 6 vertices of two triangles in the plane are in general position, then the boundaries of the triangles intersect at an even number of points.

(b) Any two closed polygonal lines in the plane whose vertices are in general position intersect at an even number of points.\footnote{This is not trivial because the polygonal lines may have self-intersections and because the Jordan Theorem is not obvious.
It is not reasonable to deduce the Parity Lemma from the Jordan Theorem or the Euler Formula because this could form a vicious circle.}
\end{lemma}


\begin{proof}[Proof of (a)]\footnote{If we prove that the triangle splits the plane into two parts (this case of the Jordan Theorem is easy), then part (a) would follow because the boundary of one triangle comes {\it into} the other triangle as many times as it comes {\it out}.
The following proof can be generalized to higher dimensions \cite{Sk14}.}
The intersection of the convex hull of one triangle and the boundary of the other triangle is a finite  union of polygonal lines (non-degenerate to points).
The boundaries of the triangles intersect at the endpoints of the polygonal lines.
The number of endpoints is even, so the fact follows.
\end{proof}

\begin{figure}[h]\centering
\includegraphics{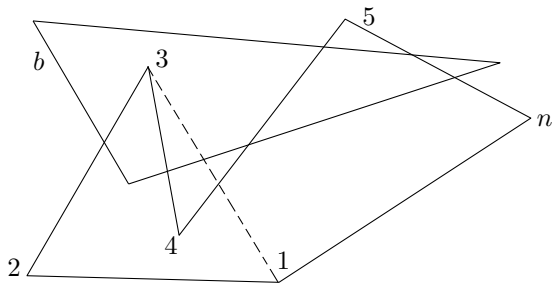}
\caption{Proof of the Parity Lemma \ref{0-even}.b}
\label{f:cut}
\end{figure}

\begin{proof}[Proof of (b)]
First assume that {\it one of the polygonal lines, say $b$, is the boundary of a triangle.}
\footnote{Part (b) is proved analogously to (a) under this assumption.
We present a different proof of this `intermediate' case, reducing it to (a).
This generalizes to a proof of the general case.}
Denote by $1,2,\ldots,n$ the consecutive vertices of the other polygonal line.
Let us prove the lemma by induction on $n\ge3$.
For the proof of the inductive step denote by $a_1\ldots a_k$ the closed polygonal line with consecutive vertices  $a_1,\ldots,a_k$.
Then (fig. \ref{f:cut})
$$
|1234\ldots n\cap b|\underset2\equiv |123\cap b|+|134\ldots n\cap b|\underset2\equiv 0.
$$
Here the second equality follows by (a)
and the inductive hypothesis.

{\it The general case} is reduced to the above particular case analogously to the above reduction of the particular case to (a).
Just replace $b$ by the second polygonal line.
\end{proof}

\begin{figure}[h]\centering
\includegraphics{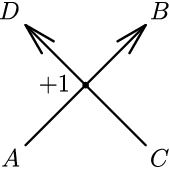}\quad \includegraphics{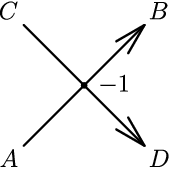}
\caption{The sign of intersection point}
\label{f:sign}
\end{figure}

Let $A,B,C,D$ be points in the plane, of which no three belong to a line.
Define {\bf the sign} of intersection point of oriented segments $\overrightarrow{AB}$ and $\overrightarrow{CD}$ as the number $+1$ if $ABC$ is oriented clockwise and the number $-1$ otherwise (fig.~\ref{f:sign} and \ref{f:gl2}).

\begin{figure}[h]
\centerline{\includegraphics[width=4.5cm]{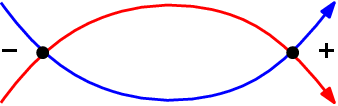} \qquad \includegraphics{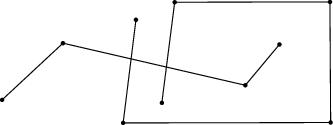}}
\caption{Two curves intersecting at an even number of points the sum of whose signs is zero (left) or non-zero (right)} \label{f:gl2}
\end{figure}


The following lemma
is proved analogously to the Parity Lemma \ref{0-even}.

\begin{lemma}[Triviality]\label{l:triv}
For any two closed polygonal lines in the plane whose vertices are in general position the sum of signs of their intersection points is zero.
\end{lemma}

The rest of this subsection is formally not used later.

\begin{pr}[see proof in \S\ref{s:appgr}]\label{pla-kon}
Take 14 general position points in the plane, of which 7 are red and another 7 are yellow.

(a) Then the number of intersection points of the red segments (i.e. the segments joining the red points) with the yellow segments is even.

(b) Electric current flows through every red segment.
The sum of the currents flowing to any red point equals the sum of the currents issuing out of the point.
The current also flows through the yellow segments conforming to the same Kirchhoff's law.
Let us orient each red or yellow segment accordingly to the direction of the current passing through it.
Assign to each intersection point of oriented red and yellow segments the product of currents passing through these segments and the sign of the intersection point.
Then the sum of all assigned products (i.e. {\it the flow of the red current through the yellow one}) is zero.
\end{pr}


\begin{remark}[on generalization to cycles]\label{r:1cyc}
(a) The Parity Lemma \ref{0-even} and Assertion \ref{pla-kon}.a have the following common generalization.
A {\it 1-cycle} (modulo 2) is a finite collection of segments in the plane such that every point of the plane is the end of an even number of them.
Then any two 1-cycles in the plane whose vertices are in general position intersect at an even number of points.

(b) A {\it 2-cycle} (modulo 2) is a finite collection of triangles in the plane such that every segment in the plane is the side of an even number of them.
If a point and vertices of triangles of a 2-cycle are in general position, then the point belongs to an even number of the triangles.
\end{remark}


\subsection{Self-intersection invariant for graph drawings}\label{0grapl}

We shall consider plane drawings of a graph such that the edges are drawn as polygonal lines and intersections are allowed.
Let us formalize this for graph $K_n$ (formalization for arbitrary graphs is presented at the beginning of \S\ref{s:vkam2-ic}).

{\bf A piecewise-linear (PL) map $f:K_n\to\R^2$} of the graph $K_n$ to the plane is a collection of ${n\choose2}$ (non-closed) polygonal lines pairwise joining some $n$ points in the plane.
{\bf The image $f(\sigma)$ of edge $\sigma$} is the corresponding polygonal line.
{\bf The image of a collection of edges} is the union of images of all the edges from the collection.

\begin{theorem}[Cf. Proposition \ref{0-ra2}.a and Theorems \ref{ratv-totv}, \ref{t:tvkf}] \label{grapl-nonalm}
For any PL (or continuous) map $K_5\to\R^2$
there are two non-adjacent edges whose images intersect.
\end{theorem}


Theorem \ref{grapl-nonalm} is deduced from its `quantitative version': for `almost every' drawing of $K_5$ in the plane the number of intersection points of non-adjacent edges is odd.
The words `almost every' are formalized below in Lemma \ref{11-vankam}.
Formally, Theorem \ref{grapl-nonalm} follows by Lemma \ref{11-vankam} using a version of \cite[Approximation Lemma 1.4.6]{Sk20}, cf. Remark \ref{r:ae}.c.


Let $f:K_n\to\R^2$ be a PL map.
It is called {\bf a general position} PL map if all the vertices of the polygonal lines are in general position.
Then by Proposition~\ref{grapl-gp} the images  of any two non-adjacent edges intersect by a finite number of points.
Let {\bf the van Kampen number} (or the self-intersection invariant) $v(f)\in\Z_2$ be the parity of the number of
all such intersection points, for all pairs of non-adjacent edges.

\begin{example}\label{1-vankam} (a) A convex pentagon with the diagonals forms a general position PL map $f:K_5\to\R^2$ such that $v(f)=1$.

(b) A convex quadrilateral with the diagonals forms a general position PL map $f:K_4\to\R^2$ such that $v(f)=1$.
A triangle and a point inside forms a general position PL map $f:K_4\to\R^2$ such that $v(f)=0$.
Cf. \S\ref{s:ratvpl} and \S\ref{0-ratvtopl}.
\end{example}

\begin{figure}[h]\centering
\includegraphics{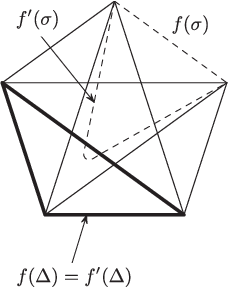}
\caption{The residue $v(f)$ is independent of $f$}
\label{k5move}
\end{figure}

\begin{lemma}[Cf. Proposition \ref{0-ra2}.b and Lemma \ref{ratv-vk2}]\label{11-vankam}
For any general position PL map $f:K_5\to\R^2$ the van Kampen number $v(f)$ is odd.
\end{lemma}

\begin{proof} By Proposition \ref{0-ra2}.b it suffices to prove that $v(f)=v(f')$ for each two general position PL maps $f,f':K_5\to\R^2$ coinciding on every edge except an edge $\sigma$, and such that $f|_\sigma$ is linear (fig.~\ref{k5move}).
The edges of $K_5$ non-adjacent to $\sigma$ form a cycle (this very property of $K_5$ is necessary for the proof). Denote this cycle by $\Delta$.
Then
$$
v(f)-v(f')=|(f\sigma\cup f'\sigma)\cap f\Delta|\mod2=0.
$$
Here the second equality follows by the Parity Lemma \ref{0-even}.b.
\end{proof}

\begin{remark}\label{r:grapl-k33} (a) The analogues of Theorem \ref{grapl-nonalm} and Lemma \ref{11-vankam} for $K_5$ replaced by $K_{3,3}$ are correct, cf. Remark \ref{grapl-k33}.

(b) Non-planarity of $K_{3,3}$ also follows by Theorem \ref{grapl-nonalm} and fig.  \ref{fig-almemb} \cite{Sk03}.
\end{remark}

\begin{figure}[h]\centering
\includegraphics[scale=.7]{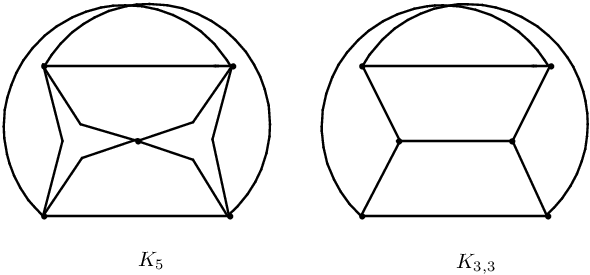}
\caption{`Almost embedding' $K_5\to K_{3,3}$}\label{fig-almemb}
\end{figure}

\subsection{A polynomial algorithm for recognizing graph planarity}\label{0vkam2}

\subsubsection{Van Kampen-Hanani-Tutte planarity criterion}\label{s:grapl-vkht}

A polynomial algorithm for recognizing graph planarity is obtained using the van Kampen-Hanani-Tutte planarity criterion (Proposition~\ref{t:pl-vk} below).
In the following subsections we show how to invent and prove that criterion.
We consider a natural object (intersection cocycle) for any general position PL map from a graph to the plane (\S\ref{s:vkam2-ic}).
Then we investigate how this object depends on the map (Proposition \ref{elcob}.b below).
So we derive from this object an obstruction to planarity which is independent of the map.
Combinatorial and linear algebraic (=cohomological) interpretation of this obstruction gives the required planarity criterion.


\begin{proposition}\label{t:pl-vk} Take any ordering of the vertices of a graph.
Then the following conditions are equivalent.

(i) The graph is planar.

(ii) There are vertices $V_1,\ldots,V_s$ and edges $e_1,\ldots,e_s$ such that $V_i\not\in e_i$ for any  $i=1,\ldots,s$, and for any non-adjacent edges $\sigma,\tau$ the following numbers have the same parity:

$\bullet$ the number of those endpoints of $\sigma$ that lie between the endpoints of $\tau$ (for the above ordering; the parity of this number is one if the endpoints of edges are `intertwined' and is zero otherwise).

$\bullet$ the number of those $i=1,\ldots,s$ for which either $V_i\in\sigma$ and $e_i=\tau$, or $V_i\in\tau$ and $e_i=\sigma$.

(iii) The following system of linear equations over $\Z_2$ is solvable.
To every pair $A,e$ of a vertex and an edge such that $A\notin e$ assign a variable $x_{A,e}$.
For every non-ordered pair of non-adjacent edges $\sigma,\tau$ denote by $b_{\sigma,\tau}\in\Z_2$ the number of
those endpoints of $\sigma$ whose numbers lie between the numbers of the endpoints of $\tau$.
For every such pairs $(A,e)$ and $\{\sigma,\tau\}$ let\footnote{Example \ref{vkam2-line} and Proposition \ref{elcob}.b explain how $b_{\sigma,\tau}$ and $a_{A,e,\sigma,\tau}$ naturally appear in the proof.}
$$a_{A,e,\sigma,\tau}=\begin{cases} 1 & \text{either}\quad(A\in\sigma\text{ and }e=\tau)\quad\text{or}\quad  (A\in\tau \text{ and }e=\sigma)\\ 0 & \text{otherwise}\end{cases}.$$
For every such pair $\{\sigma,\tau\}$ take the  equation
$\sum_{A\notin e} a_{A,e,\sigma,\tau}x_{A,e}=b_{\sigma,\tau}$.
\end{proposition}

The implication $(ii)\Leftrightarrow(iii)$ is clear.
The implication $(ii)\Rightarrow(i)$ follows by the Kuratowski Theorem \ref{grapl-kur}
and Assertion \ref{grapl-wives} below.
The implication $(i)\Rightarrow(iii)$ follows by the Hanani-Tutte Theorem \ref{vkam-z2},
Example \ref{vkam2-line} and Proposition \ref{starpr} below.

For history see \cite[Remark after Theorem 1.18]{Sc13} and Remark \ref{r:histgr}.
For generalization to surfaces see the survey \cite{Sk21m} and the references therein.

\begin{pr}\jonly{[see proof in {\cite[\S1.5.1]{Sk18}}]} \label{grapl-wives}
The property (ii) above is not fulfilled for $K_5$ and for $K_{3,3}$.

Let us present a direct reformulation for $K_5$ (for $K_{3,3}$ the reformulation and the proof are analogous).

There are five musicians of different age.
Some pairs of musicians performed pieces.
Every pair was listened by some (possibly by none) of the remaining three musicians.
Then for some two disjoint pairs of musicians the sum of the following three numbers is odd:

$\bullet$ the number of musicians from the first pair whose age is between the ages of musicians from the second pair,

$\bullet$ the number of musicians from the first pair who listened to the second pair,

$\bullet$ the number of musicians from the second pair who listened to the first pair.

Here is restatement in mathematical language.

Let $A_1,\ldots,A_5$ be five collections of 2-element subsets of $\{1,2,3,4,5\}$ such that no $j\in\{1,2,3,4,5\}$ is contained in any subset from $A_j$.
Then for some four different elements $i,j,k,l\in\{1,2,3,4,5\}$ the sum of the following three numbers is odd

$\bullet$ the number of elements $s\in\{i,j\}$ lying between $k$ and $l$;

$\bullet$ the number of elements $s\in\{i,j\}$ such that $A_s\ni\{k,l\}$;

$\bullet$ the number of elements $s\in\{k,l\}$ such that $A_s\ni\{i,j\}$.
\end{pr}


\begin{proof}[Sketch of proof]  Denote by $S$ the sum of the considered sums over all 15 unordered pairs of disjoint pairs of musicians. One can check that $S$ is odd for any choice of performances. See a `geometric interpretation' in Example \ref{cobo}.b.
\end{proof}


\subsubsection{Intersection cocycle}\label{s:vkam2-ic}

{\it A linear map $f:K\to\R^2$} of a graph $K=(V,E)$ to the plane is a map $f:V\to\R^2$.
{\it The image $f(AB)$ of edge $AB$} is the segment $f(A)f(B)$.
{\bf A piecewise-linear (PL) map $f:K\to\R^2$} of a graph $K=(V,E)$ to the plane is a collection of (non-closed) polygonal lines corresponding to the edges of $K$, whose endpoints correspond to the vertices of $K$.
(A PL map of a graph $K$ to the plane is `the same' as a linear map of some graph homeomorphic to $K$.)
{\bf The image} of an edge, or  of a collection of edges, is defined analogously to the case of $K_n$ (\S\ref{0grapl}).
So a graph is planar if there exists its PL map to the plane such that the images of vertices are distinct,
the images of the edges do not have self-intersections, and no image of an edge intersects the interior of any other image of an edge.

A linear map of a graph to the plane is called {\it a general position linear map} if the images of all the vertices are in general position.
A PL map $f:K\to\R^2$ of a graph $K$ is called {\bf a general position} PL map if there exist a graph $H$ homeomorphic to $K$ and a general position linear map of $H$ to the plane such that this map `corresponds' to the map $f$.


A graph is called {\bf $\Z_2$-planar} if there exists a general position PL map of this graph to the plane such that images of any two non-adjacent edges intersect at an even number of points.


By Lemma~\ref{11-vankam} $K_5$ is not $\Z_2$-planar.
Analogously, $K_{3,3}$ is not $\Z_2$-planar (see Remark \ref{r:grapl-k33}.a).
Hence, if a graph $K$ is homeomorphic to $K_5$ or to $K_{3,3}$, then $K$ is not $\Z_2$-planar (because any PL map $K\to\R^2$ corresponds to some PL map $K_5\to\R^2$ or $K_{3,3}\to\R^2$).
Then using Kuratowski Theorem \ref{grapl-kur}
one can obtain the following result.\footnote{For a direct deduction of planarity from $\Z_2$-planarity see~\cite{Sa91}; K. Sarkaria confirms that the deduction has gaps.
For a direct deduction of $\Z_2$-planarity from non-existence of a subgraph homeomorphic to $K_5$ or $K_{3,3}$ see~\cite{Sa91}.}

\begin{theorem}[Hanani-Tutte; cf. Theorems \ref{vkam-z2-r} and \ref{t:mawa}]\label{vkam-z2}
A graph is planar if and only if it is $\Z_2$-planar.
\end{theorem}

\begin{example}\label{vkam2-line}
Suppose a graph and an arbitrary ordering of its vertices are given.
Put the vertices on a circle, preserving the ordering.
Take the chord for each edge.
We obtain a general position linear map of the graph to the plane.
For any non-adjacent edges $\sigma,\tau$ the number of intersection points of their images has the same parity as the number of endpoints of $\sigma$ that lie between the endpoints of $\tau$.
\end{example}

Let $f:K\to\R^2$ be a general position PL map of a graph $K$.
Take any pair of non-adjacent edges $\sigma,\tau$.
By Proposition~\ref{grapl-gp} the intersection $f\sigma\cap f\tau$ consists of a finite number of points.
Assign to the pair $\{\sigma,\tau\}$ the residue
$$|f\sigma\cap f\tau|\mod2.$$
Denote by $K^*$ the set of all unordered pairs of non-adjacent edges of the graph $K$.
The obtained map $K^*\to\Z_2$ is called {\bf the intersection cocycle} (modulo 2) of $f$ (we call it `cocycle' instead of `map' to avoid confusion with maps to the plane).
Maps $K^*\to\Z_2$ are identified with subsets of $K^*$ (consisting of pairs going to $1\in\Z_2$).\footnote{Maps $K^*\to\Z_2$ can also be seen as `partial matrices', i.e. symmetric arrangements of zeroes and ones in those cells of the $e\times e$-matrix that correspond to the pairs of non-adjacent edges, where $e$ is the number of edges of $K$.}


\begin{remark}\label{vector} (a) If a graph is $\Z_2$-planar, then the intersection cocycle is zero for some general position PL map of this graph to the plane.

(b) (cf. Example \ref{vkam2-line})
Take a linear map $f:K_n\to\R^2$ such that $f(1)f(2)\ldots f(n)$ is a convex $n$-gon.
For $n=4$ and $n=5$ the intersection cocycles correspond to the subsets $\big\{ \{13,24\} \big\}$  and
  $\big\{\{13,24\},\{24,35\},\{35,41\},\{41,52\},\{52,13\}\big\}$.
We obtain the following partial matrices (the edges are ordered lexicographically).
$$\left(  \begin{array}{cccccc}
    \text{-} & \text{-} & \text{-} & \text{-} & \text{-} & 0 \\
    \text{-} & \text{-} & \text{-} & \text{-} & 1 & \text{-} \\
    \text{-} & \text{-} & \text{-} & 0 & \text{-} & \text{-} \\
    \text{-} & \text{-} & 0 & \text{-} & \text{-} & \text{-} \\
    \text{-} & 1 & \text{-} & \text{-} & \text{-} & \text{-} \\
    0 & \text{-} & \text{-} & \text{-} & \text{-} & \text{-} \\
  \end{array}\right)
\quad\text{and}\quad
\left(  \begin{array}{cccccccccc}
    \text{-} & \text{-} & \text{-} & \text{-} & \text{-} & \text{-} & \text{-} & 0 & 0 & 0 \\
    \text{-} & \text{-} & \text{-} & \text{-} & \text{-} & 1 & 1 & \text{-} & \text{-} & 0 \\
    \text{-} & \text{-} & \text{-} & \text{-} & 0 & \text{-} & 1 & \text{-} & 1 & \text{-} \\
    \text{-} & \text{-} & \text{-} & \text{-} & 0 & 0 & \text{-} & 0 & \text{-} & \text{-} \\
    \text{-} & \text{-} & 0 & 0 & \text{-} & \text{-} & \text{-} & \text{-} & \text{-} & 0 \\
    \text{-} & 1 & \text{-} & 0 & \text{-} & \text{-} & \text{-} & \text{-} & 1 & \text{-} \\
    \text{-} & 1 & 1 & \text{-} & \text{-} & \text{-} & \text{-} & 0 & \text{-} & \text{-} \\
    0 & \text{-} & \text{-} & 0 & \text{-} & \text{-} & 0 & \text{-} & \text{-} & \text{-} \\
    0 & \text{-} & 1 & \text{-} & \text{-} & 1 & \text{-} & \text{-} & \text{-} & \text{-} \\
    0 & 0 & \text{-} & \text{-} & 0 & \text{-} & \text{-} & \text{-} & \text{-} & \text{-} \\
  \end{array}\right).$$

(c) A subset $C\subset K^*$ is called a {\bf 2-cycle} (modulo 2) if for each edge $\sigma$ and vertex $A\not\in\sigma$ there is an even number of edges $\tau$ having a vertex $A$ and such that $\{\sigma,\tau\}\in C$.
For a general position PL map $f:K\to\R^2$ define the {\it $C$-van Kampen number}
$$v_C(f):=\sum\limits_{ \{\sigma,\tau\} \in C} |f\sigma\cap f\tau| \in\Z_2.$$
Analogously to Lemma \ref{ratv-vk2} $v_C(f)$ is independent of $f$, and so depends only on $K$ and $C$.

{\it A graph $K$ is $\Z_2$-planar if and only if the its $C$-van Kampen number is zero
for any 2-cycle $C\subset K^*$.}

This can be deduced from  Proposition \ref{starpr}.
\end{remark}

\subsubsection{Intersection cocycles of different maps}\label{s:vkam2-cc}

Addition of cocycles $K^*\to\Z_2$ is componentwise, i.e. is defined by adding modulo 2 numbers corresponding to the same pair.
This corresponds to the sum modulo 2 of subsets of $K^*$.

\begin{figure}[h]\centering
\includegraphics{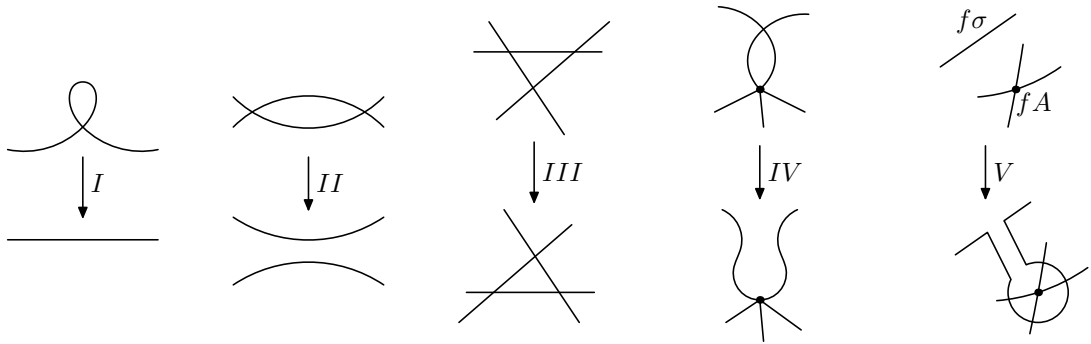}
\caption{The Reidemeister moves for graphs in the plane}
\label{reidall}
\end{figure}


\begin{proposition}[cf. Proposition \ref{elcobz}]\label{elcob}
(a) The intersection cocycle does not change under the first four Reidemeister moves in fig.~\ref{reidall}.I-IV.
(The graph drawing changes in the disk as in fig.~\ref{reidall}, while out of this disk the graph drawing remains unchanged.
In (a) no other images of edges besides the pictured ones intersect the disk.)

(b) Let $K$ be a graph and $A$ its vertex which is not the end of an edge $\sigma$.
{\bf An elementary coboundary} of the pair $(A,\sigma)$ is the subset $\delta_K(A,\sigma)\subset K^*$
consisting of all pairs $\{\sigma,\tau\}$ with $\tau\ni A$.
Under the Reidemeister move in Fig.~\ref{reidall}.V the intersection cocycle changes by adding $\delta_K(A,\sigma)$.
(In (b) other images of edges besides the pictured ones intersect the disk, but <<parallel>> segments are <<very close>> to each other.)
\end{proposition}


\begin{example}\label{cobo} The  subset of $\delta_K(A,\sigma)^{-1}(1)\subset K^*$ corresponding to the cocycle  $\delta_K(A,\sigma)$ is also called elementary coboundary.

(a) We have $\big\{\{13,24\}\big\}=\delta_{K_4}(1,24)=\delta_{K_4}(2,13)=\delta_{K_4}(3,24)= \delta_{K_4}(4,13)$.
So the intersection cocycle of Remark~\ref{vector}.b is an elementary coboundary for $n=4$.

(b) We have $\delta_{K_5}(3,12)=\big\{\{12,34\},\{12,35\}\big\}$.
So the intersection cocycle of Remark~\ref{vector}.b for $n=5$ is not a sum of elementary coboundaries.
Indeed, addition of an elementary coboundary does not change the parity of the number of ones above the diagonal, while initially this number is 5.
Thus we proved Assertion \ref{grapl-wives} for $K_5$.
\end{example}

Cocycles $\nu_1,\nu_2:K^*\to\Z_2$ (or $\nu_1,\nu_2\subset K^*$) are called {\bf cohomologous} if
$$
\nu_1-\nu_2=\delta_K(A_1,\sigma_1)+\ldots+\delta_K(A_k,\sigma_k)
$$
for some vertices $A_1,\ldots,A_k$ and edges $\sigma_1,\ldots,\sigma_k$ (not necessarily distinct).

Proposition~\ref{elcob}.b and the following Lemma~\ref{star} show that cohomology is the equivalence relation generated by changes of a graph drawing.

\begin{lemma}[cf. Lemmas \ref{star-zl} and \ref{star-r}]\label{star} The intersection cocycles of different general position PL maps of the same graph to the plane are cohomologous.\footnote{This lemma is implied by, but is easier than, the following fact: any general position PL map of a graph to the plane can be obtained from any other such map using Reidemeister moves in fig.~\ref{reidall}.}
\end{lemma}


The proof of Lemma \ref{star} presented below \jonly{in \cite[\S1.5]{Sk18}} is a non-trivial generalization of the proof of Lemma \ref{11-vankam}.
Lemma~\ref{star} and Proposition~\ref{elcob}.b imply the following result.

\begin{proposition}[cf. Propositions \ref{starz} and \ref{starpr-r}]\label{starpr}
A graph is $\Z_2$-planar if and only if the intersection cocycle of some (or, equivalently, of any)
general position PL map of this graph to the plane is cohomologous to the zero cocycle.
\end{proposition}

\begin{proof}[Proof of Lemma~\ref{star}]
Suppose that $K$ is a graph and $f,f':K\to\R^2$ are general position PL maps.

\noindent {\it Proof of the particular case when the maps $f$ and $f'$ differ only on  the interior of one edge~$\sigma$.}
It suffices to prove this case when $f|_\sigma$ is linear.
Take a point $O$ in the plane in general position together with all the vertices of the polygonal lines $f\alpha$ and $f'\alpha$.
Denote by $\overline X$ the residue $|(f\sigma\cup f'\sigma)\cap X| \mod 2$.
Then by the Parity Lemma~\ref{0-even}.b for every edge $PQ$ disjoint with $\sigma$ we have
$$0 = \overline{Of(P)\cup Of(Q)\cup f(PQ)} = \overline{Of(P)}+\overline{Of(Q)}+\overline{f(PQ)}\quad\Rightarrow\quad \overline{f(PQ)} = \overline{Of(P)} + \overline{Of(Q)}.$$
(Cf. Proposition \ref{ratv-chess}.ab.)
Then the difference of the intersection cocycles
of $f$ and
of $f'$ equals
$$\sum\limits_{PQ\cap\sigma=\emptyset} \overline{f(PQ)} =
\sum\limits_{PQ\cap\sigma=\emptyset}  (\overline{Of(P)}+\overline{Of(Q)})=
\sum\limits_{B\not\in \sigma} \overline{Of(B)}\cdot\delta(B,\sigma).$$
(This number is equal $\delta(B_1,\sigma)+\ldots+\delta(B_k,\sigma)$, where $B_1,\ldots,B_k$ all the vertices $B\not\in\sigma$ for which the segment $Of(B)$ intersect the cycle $f\sigma\cup f'\sigma$ at an odd number of points; the set $B_1,\ldots,B_k$ depends on the choice of the point $O$, but the equality holds for every choice.)

\begin{figure}[h]\centering
\includegraphics{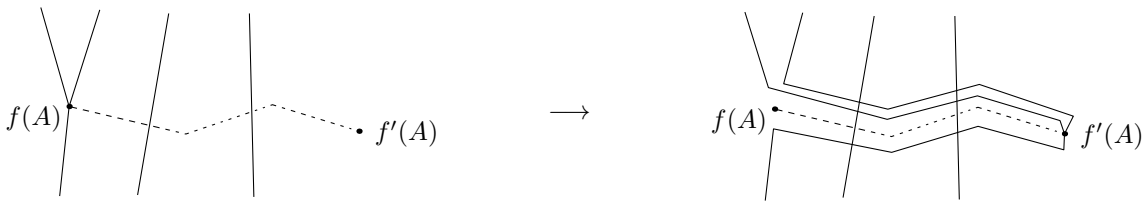}
\caption{Transformation of a general position PL map}
\label{f:11-vankam}
\end{figure}

\noindent {\it Deduction of the general case from the particular case (suggested by R. Karasev).}
It suffices to prove the lemma for map $f$ that differs from $f'$ only on the set of edges issuing out of some vertex $A$.
Join $f(A)$ to $f'(A)$ by a general position polygonal line.
Change the maps $f$ and $f'$ on the interiors of the edges so that this polygonal line does not intersect the $f$- and $f'$-images of the edges issuing out of $A$.
By the above particular case, the intersection cocycles will changes to cohomologous ones.
Then take a map $f''$ obtained from $f$ by `moving the neighborhood of $f(A)$ to the $f'(A)$ along the polygonal line' (fig.~\ref{f:11-vankam} which is a version of fig. \ref{reidall}.V).
The intersection cocycles $\nu(f)$ and $\nu(f'')$ are cohomologous.
By the above particular case, the intersection cocycles $\nu(f')$ and $\nu(f'')$ are cohomologous.
Therefore $\nu(f)$ and $\nu(f')$ are cohomologous.
\end{proof}

\begin{proof}[Another proof of Lemma~\ref{star}]
(This proof can be generalized to higher dimensions and perhaps to higher multiplicity, see Problem \ref{p:boundary}.)
Let $K$ be the graph.
Take a {\it general position PL homotopy} $f_t:K\to\R^2$, $t\in[0,1]$, between given two general position PL maps $f_0$ and $f_1$.
Let $\{A_1,\sigma_1\},\ldots,\{A_k,\sigma_k\}$ be all pairs $(A,\sigma)$ of a vertex and an edge not containing this vertex such that the number of $t\in[0,1]$ with $f_tA\in f_t\sigma$ is odd.
It suffices to prove that the difference of the intersection cocycles of $f_0$ and of $f_1$ equals
$$\delta_K(A_1,\sigma_1)+\ldots+\delta_K(A_k,\sigma_k).$$
\begin{figure}[h]\centering
\includegraphics{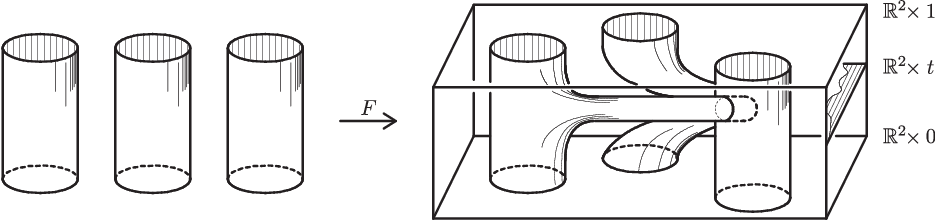}
\caption{A homotopy}
\label{f:homotopy}
\end{figure}

Let us prove this equality.
The homotopy $f_t$ can be seen (fig. \ref{f:homotopy}) as a `general position' PL map $F:K\times[0,1]\to\R^2\times[0,1]$ such that
$$F(K\times t)\subset\R^2\times t\quad\text{for each}\quad t\in[0,1]
\quad\text{and}\quad  f(x,t)=f_t(x)\quad\text{for each}\quad t=0,1.$$
By general position for each two disjoint edges $\sigma,\tau$ of $K$ the intersection
$C:=F(\sigma\times[0,1])\cap F(\tau\times[0,1])$ is a finite union of non-degenerate segments.
Now the required equality follows because
$$
0 \underset2\equiv |\partial C| \underset2\equiv |f_0\sigma\cap f_0\tau| + |f_1\sigma\cap f_1\tau| + |F(\partial\sigma\times[0,1])\cap F(\tau\times[0,1])| + |F(\sigma\times[0,1])\cap F(\partial\tau\times[0,1])|.
$$
\end{proof}

Denote the set of cocycles $K^*\to\Z_2$ up to cohomology by $H^2(K^*;\Z_2)$.
(This is called {\it two-dimensional cohomology group of $K^*$ with coefficients in $\Z_2$}.)
A cohomology class of the intersection cocycle of some (or, equivalently, of any) general position PL map $f:K\to\R^2$ is called {\it the van Kampen obstruction modulo 2} $v(K)\in H^2(K^*;\Z_2)$.
Lemma~\ref{star} and Proposition~\ref{starpr} are then reformulated as follows:

$\bullet$ the class $v(K)$ is well-defined, i.e. it does not depend of the choice of the map $f$.

$\bullet$ a graph $K$ is $\Z_2$-planar if and only if $v(K)=0$.

\subsubsection{Intersections with signs}\label{0vkamz}

Here we generalize previous constructions from residues modulo 2 to integers.
This generalization is not formally used later.
However, it is useful to make this simple generalization (and possibly to make
Remark \ref{vkam-line}) before more complicated generalizations in \S\ref{s:triple}, \S\ref{s:tvkam}.
Also, integer analogues are required for higher dimensions (namely, for  proofs of Theorems \ref{t:rec}, \ref{t:algal}).

Suppose that $P$ and $Q$ are oriented polygonal lines in the plane whose vertices are in general position.
Define the {\bf algebraic intersection number} $P\cdot Q$ of $P$ and $Q$ as the sum of the signs of the intersection points of $P$ and $Q$.
See fig. \ref{f:gl2}.

\begin{pr}\label{starzch}
(a) We have $P\cdot Q=-Q\cdot P$.

(b) If we change the orientation of $P$, then the sign of $P\cdot Q$ will change.


(c) If we change the orientation of the plane, i.e. if we make axial symmetry, then
the sign of $P\cdot Q$ will change.
\end{pr}

Let $K$ be a graph and $f:K\to\R^2$ a general position PL map.
Orient the edges of $K$.
Assign to every ordered pair $(\sigma,\tau)$ of non-adjacent edges the algebraic intersection number  $f\sigma\cdot f\tau$.
Denote by $\widetilde  K$ the set of all ordered pairs of non-adjacent edges of $K$.
The obtained cocycle (=map) $\cdot:\widetilde  K\to\Z$ is called {\bf the integral intersection cocycle} of $f$ (for given orientations).

\begin{proposition}\label{elcobz}
Analogue of Proposition \ref{elcob} is true for the integral intersection cocycle, with the following definition.
Let $K$ be an oriented graph and $A$ a vertex which is not the end of an edge $\sigma$.
{\bf An elementary skew-symmetric coboundary} of the pair $(A,\sigma)$
is the cocycle $\delta_{K,\Z}(A,\sigma):\widetilde  K\to\Z$ that assigns

$\bullet$ $+1$ to any pair $(\sigma,\tau)$ with $\tau$ issuing out of $A$ and any pair $(\tau,\sigma)$ with $\tau$ going to $A$,

$\bullet$ $-1$ to any pair $(\sigma,\tau)$ with $\tau$ going to $A$ and any pair $(\tau,\sigma)$ with $\tau$ issuing out~of~$A$,

$\bullet$ $0$ to any other pair.
\end{proposition}


Two cocycles $N_1,N_2:\widetilde  K\to\Z$ are called {\bf skew-symmetrically cohomologous}, if
$$
N_1-N_2=m_1\delta_{K,\Z}(A_1,\sigma_1)+\ldots+m_k\delta_{K,\Z}(A_k,\sigma_k)
$$
for some vertices $A_1,\ldots,A_k$, edges $\sigma_1,\ldots,\sigma_k$ and integers $m_1,\ldots,m_k$ (not necessarily distinct).

The following integral analogue of Lemma \ref{star} is proved analogously using the Triviality Lemma \ref{l:triv}.

\begin{lemma}\label{star-zl} The integer intersection cocycles of different maps of the same graph to the plane are skew-symmetrically cohomologous.
\end{lemma}


\begin{proposition}[cf. Proposition \ref{starz-r}]\label{starz}
Twice the integral intersection cocycle of any general position PL map of a graph in the plane is skew-symmetrically cohomologous to the zero cocycle.
\end{proposition}

This follows by Assertion \ref{starzch}.c and Lemma \ref{star-zl}.

Denote the set of skew-symmetric (cf. Assertion \ref{starzch}.a) cocycles $\widetilde  K\to\Z$ up to skew-symmetric cohomology by $H^2_{ss}(\widetilde  K;\Z)$.
A skew-symmetric cohomology class of the integer intersection cocycle of some (or, equivalently, of any) general position PL map $f:K\to\R^2$ is called {\it the van Kampen obstruction} $V(K)\in H^2_{ss}(\widetilde  K;\Z)$, see
Remark
\ref{r:me-fkt}.

\subsection{Appendix: some details to \S\ref{0-gra}}\label{s:appgr}

See an alternative proof of Proposition \ref{0-ra2}.ab in \cite[\S2]{Sk14}.
Proposition \ref{0-ra2}.c is proved analogously.

\begin{proof}[Proof of the General Position Theorem \ref{0-gp3}]
Choose three points in 3-space that do not belong to one line.
Suppose that we have $n\geq3$ points in general position.
Then there is a finite number of planes containing triples of these $n$ points.
Hence there is a point that does not lie on any of these planes.
Add this point to our set of $n$ points.
Since  the `new' point is not in one plane with any three of the `old' $n$ points, the obtained set of $n+1$ points is in general position.
Thus {\it for each $n$ there exist $n$ points in 3-space that are in general position.}

Take such $n$ points.
Denote by $A$ the set of all segments joining pairs of these points.
If some two segments from $A$ with different endpoints intersect, then four endpoints of these two segments lie in one plane.
If some two segments from $A$ with common endpoint intersect not only at their common endpoint, then the three endpoints of these two segments are on one line.
So we obtain a contradiction.
\end{proof}


\begin{proof}[Proof of Proposition \ref{1-k5-1}]
Let us prove part (a), other parts are proved analogously.

Any five points can be transformed into five points in general position leaving the required properties unchanged.
By hypothesis, the number of intersection points of segment 12 and the boundary of triangle 345 equals to the number of intersection points of interiors of segments joining the points.
This number is odd by Proposition \ref{0-ra2}.b.
\end{proof}

Analogously to Proposition \ref{1-k5-1}.a one proves the following proposition.
(One can also state and prove PL analogues of Propositions \ref{1-k5-1}.bcd.
For the proof instead of Proposition \ref{0-ra2}.b one would need Lemma \ref{11-vankam}.)

A {\it PL embedding} (or PL realization) of a graph in the plane is a set of points and polygonal lines
from the definition of planarity such that no isolated vertex lies on any of the polygonal lines.
The points are called the images of vertices, and the polygonal lines are called the images of edges.

\begin{proposition}\label{grapl-ram} Remove the edge joining the vertices 1 and 2 form the graph $K_5$.
Then for every PL embedding of the obtained graph in the plane any polygonal line joining the images of the vertices 1 and 2 intersects the image of the cycle 345 (i.e. the images of the vertices 1 and 2 are {\it separated} by the cycle 345).
\end{proposition}


\begin{proof}[Sketch of proof of Proposition \ref{1-alg}]
The proof is based on the following notion of isotopy and lemma.
Two general position subsets $M$ and $M'$ in
of the plane are called {\it isotopic} if there is a bijection $f:M\to M'$ such that for any $A,B,C,D\in M$
the segment $AB$ intersects the line $CD$
if only if the segment $f(A)f(B)$ intersects the line $f(C)f(D)$.




{\it Lemma.} For any $n$ there is a finite number of $n$-element general position subsets of the plane such that every $n$-element general position subset of the plane is isotopic to one of them.

The proof is based on the following fact: $n$ lines split the plane into $(n^2+n+2)/2$ parts.
\end{proof}

\begin{proof}[An alternative proof of the Parity Lemma \ref{0-even}.b]
This proof uses {\it singular cone} idea which formalizes in a short way the {\it motion-to-infinity} idea \cite[\S5]{BE82}.
This proof generalizes to higher dimensions.

\begin{figure}[h]\centering
\includegraphics[scale=1.4]{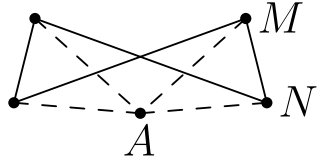}
\caption{Singular cone idea}
\label{f:sincon}
\end{figure}

First assume that one of the polygonal lines $b$ is the boundary of a triangle.
Denote another polygonal line by $a$.
Take a point $A$ such that $\partial(AMN)$ and $b$
are in general position for each edge $MN$ of the polygonal line $a$.
Denote by $\partial T$ the boundary of a triangle $T$.
Then (see fig. \ref{f:sincon})
$$
|a\cap b|= \sum\limits_{MN}|MN\cap b|\underset2\equiv\sum\limits_{MN}|\partial(AMN)\cap b|\underset2\equiv 0.
$$
Here the summation is over edges $MN$ of $a$, and the last congruence follows by the particular case of triangles.

{\it The general case} is reduced to the above particular case as in the previous proof.
\end{proof}

\begin{proof}[Proof of Assertion \ref{pla-kon}.a] (This is analogous to the above alternative proof of the Parity Lemma \ref{0-even}.b.)
Denote the yellow points by $A_1,\ldots,A_7$ and the red points by $B_1,\ldots,B_7$.
Take two points $C$ and $D$ in the plane such that all the 16 points are in general position.
Then
$$
0 \underset2\equiv \sum_{i<j,\,k<l}|\partial(CA_iA_j)\cap\partial(DB_kB_l)| \underset2\equiv \sum_{i<j,\,k<l}|A_iA_j\cap B_kB_l|.
$$
Here the first equality follows from the fact in the hint, and the second one holds because each of the edges $CA_i$ and $DB_j$ belongs to six triangles, so the intersection points lying on these edges appear in the first sum an even number of times.
\end{proof}



\begin{proof}[Proof of Assertion \ref{pla-kon}.b]
Analogously to (a) this is deduced from the fact that for $3+3$ points the flow of the red current through the yellow one equals to zero.
(This fact follows because the current is constant on edges of each triangle, and the signs of intersection points alternate, therefore the sum equals zero.)

Denote by {\it red current} (resp. yellow) assignment of currents
for red (resp. yellow) segments conforming to the Kirchhoff law.
Note that if we consider two red currents and one yellow current, then
the flow of the sum of the red currents through the yellow one equals the sum of the flows.
Analogously for one red current and two yellow currents the sum of the
flows equals the flow of the sum (that is, the flow is {\it biadditive}).

Add a point $C$ to the yellow points and a point $D$ to the red points so
that all 16 points are in general position, as in (a).
Denote the yellow points by $A_1,\ldots,A_7$ and the red points by $B_1,\ldots,B_7$.
Assume that the currents on the segments $CA_i$ and $DB_j$ equal zero.
For each segment $A_iA_j$ consider the current flowing through the triangle $CA_iA_j$ equal
the initial yellow current through $A_iA_j$ (and equal zero out of the triangle $CA_iA_j$).
Then the sum of these $7\choose2$ currents equals the initial yellow current.
Analogously decompose the red current into the sum of $7\choose2$ currents flowing through triangles $DB_kB_l$.
The required statement follows from the biadditivity and analogue for 3+3 points.
\end{proof}

\begin{remark}\label{grapl-k33} {\it The van Kampen number} of a general position PL map $f:K\to\R^2$ of a graph $K$ is defined analogously to the case $K=K_n$ (\S\ref{0grapl}).

(a) Clearly, if a graph $K$ is planar then $v(f)=0$ for some general position PL map $f:K\to\R^2$.

(b) Take two segments with a common interior point in the plane.
This forms a planar graph $K$ and a general position PL map $f:K\to\R^2$ such that $v(f)\ne0$.

(c) If $K$ is a disjoint union of two triangles, then by the Parity Lemma \ref{0-even}.b $v(f)=0$ for every general position PL map $f:K\to\R^2$.

(d) For any general position PL map $f:K\to\R^2$ of a graph $K$ the van Kampen number is the sum of values of the intersection cocycle over all non-ordered pairs of disjoint edges of $K$.
\end{remark}

\begin{remark}\label{vkam-z}
A graph is called {\it $\Z$-planar} if there exists a general position PL map of this graph to the plane such that
for images of any two non-adjacent edges the sum of the signs of their intersection points is zero for some
orientations on the edges, see fig. \ref{f:gl2}.
The sign of this sum depends on the order of the edges and on an arbitrary choice of orientations on the edges, but the condition that the some is zero does not.
One can prove analogously to Theorem \ref{vkam-z2}, or deduce from it, that {\it a graph is planar if and only if it is $\Z$-planar.}
Integral analogue of Proposition \ref{starpr} is correct and follows  analogously by Proposition \ref{elcobz} and Lemma \ref{star-zl}.
\end{remark}

\begin{remark}\label{r:me-fkt}
If in \S\ref{0vkamz} we assume that cells $\sigma\times\tau$ and $\tau\times\sigma$ of $\widetilde  K$ (considered as a cell complex) are oriented coherently with the involution $(x,y)\overset{t}\leftrightarrow(y,x)$ (and so not necessarily oriented as the products), and define the intersection cochain by assigning the number
$f\sigma\cdot f\tau$ to the cell $\sigma\times\tau$ oriented as the product (and so not necessarily positively oriented), then we obtain symmetric cochains / coboundaries / cohomology and the van Kampen obstruction in the group $H^2_s(\widetilde  K;\Z)\cong H^2(K^*;\Z)$.
We have $H^2_s(\widetilde  K;\Z)\cong H^2_{ss}(\widetilde  K;\Z)$.
The two van Kampen obstructions go one to the other under this isomorphism.
Analogous remark holds for the van Kampen obstruction for embedding of $n$-complexes in $\R^{2n}$  \cite[\S3]{Sh57}, \cite[\S4.4]{Sk06}.

I am grateful to S. Melikhov for indicating that in \cite[\S2.3]{FKT} the signs are not accurate
\cite[beginning of \S1]{Me06}.
The sign error is in the fact that for $n:=\dim K$ odd and $o_f$ the integer intersection cocycle
{\it both} equalities $o_f(\sigma\times\tau)=f\sigma\cdot f\tau$ \cite[\S2.3, line 7]{FKT} and $t(\sigma\times\tau)=\tau\times\sigma$ \cite[p. 168, line -4]{FKT} for each $\sigma,\tau$ cannot be true.
If cells $\sigma\times\tau$ are oriented as the products (as in \cite[\S3]{Sh57}, \cite[\S4.4]{Sk06}), then $o_f(\sigma\times\tau)=f\sigma\cdot f\tau$ but  $t(\sigma\times\tau)=(-1)^n\tau\times\sigma$.
If cells $\sigma\times\tau$ are oriented coherently with the involution $t$ (as in \cite[\S2, Equivariant cohomology and Smith sequences]{Me06}), then
$t(\sigma\times\tau)=\tau\times\sigma$ but either $o_f(\sigma\times\tau)=-f\sigma\cdot f\tau$ or $o_f(\tau\times\sigma)=-f\tau\cdot f\sigma$.
(The orientation assumption is not explicitly introduced in \cite[\S2]{FKT}.)%
\footnote{I am grateful to V. Krushkal for helping me to locate
the sign error in \cite[\S2.3]{FKT}.
The more so because the explanation in \cite[\S3, footnote 6]{Me06} of the sign error  is confusing.
Indeed, in  \cite[\S2]{Me06} the `coherent' orientation is fixed, and without change of the orientation convention in \cite[\S3, Geometric definition of $\vartheta(X)$]{Me06} the `product' orientation is used (otherwise the formula $t(\sigma\times\tau)=(-1)^n\tau\times\sigma$ is incorrect for $n$ odd).
The sign error appears exactly because of difference between these orientation conventions.}


Definitions of the van Kampen obstruction in \S\ref{0vkamz} (and in \cite[Appendix D]{MTW}) use the product orientation on $\sigma\times\tau$ and do not mention the wrong (for $n$ odd and this orientation convention) formula $t(\sigma\times\tau)=\tau\times\sigma$.
So they do not have the sign error.
\end{remark}

\begin{remark}[Obstruction to `$\Z_2$-linearity']\label{vkam-line}
Let $K$ be a graph.
Let $f:K\to\R$ a general position PL map, i.e. a map which maps vertices to distinct points different from `return' points of edges.

(a) Clearly, for any continuous (or PL) map of a triangle to the line the image of certain vertex belongs to the image of the opposite edge.
(This is the Topological Radon theorem for the line, cf. the Topological Radon Theorems \ref{ratv-totv} and \ref{t:tr}.)
Analogous assertion holds for the triod $K_{1,3}$ instead of the triangle, and hence for any connected graph distinct from the path.
This means that the only connected `$\Z_2$-linear' graphs are paths.
However, the following discussions of a combinatorial (=cohomological) obstruction to `$\Z_2$-linearity' of a graph are interesting because they illustrate more complicated generalizations of \S\ref{s:tvkam}, and also
show how product of cochains (and of cohomology classes) appears in studies of a geometric problem, cf. \cite{Sk17d}.


(b) For any distinct points $x,y,z,t\in\R$ the following number is even:
$$|x\cap[z,t]|+|y\cap[z,t]|+|[x,y]\cap z|+|[x,y]\cap t|.$$


(c) The map assigning the number $|f(A)\cap f(BC)|$ to any pair $A,BC$ consisting of a vertex $A$ and an edge $BC$ such that $A\ne B,C$ is called {\it the intersection cocycle} of $f$.
One can define analogues of the Reidemeister moves in Fig.~\ref{reidall} for maps of a graphs to the line.
One can check how the intersection cocycle change for the analogous moves.
Then one arrives to definitions in (d) and (e) below.

(d) Define graph $K^{*(1)}$ as follows.
The vertices of $K^{*(1)}$ are unordered pairs $\{A,B\}$ of different vertices of a graph $K$.
For each pair $A,BC$ consisting of a vertex $A$ and an edge $BC$ of $K$ such that $A\ne B,C$ connect vertex $\{A,B\}$ to $\{A,C\}$ with an edge in graph $K^{*(1)}$.
Denote this edge by $A\times BC=BC\times A$.
E.g. if $K$ is the cycle $K_3$ with 3 vertices or the triod $K_{1,3}$, then $K^{*(1)}$ is the cycle with 3 or 6 vertices, respectively.


(e) For a vertex $B$ of a graph $G$ define {\it an elementary coboundary} $\delta_G B$ as a map from the set $E(G)$ of the edges of $G$ to the set $\{0,1\}$ which assigns 1 to every edge containing this vertex, and 0 to every other edge.
(In other words, $\delta_G B$ corresponds to the set of all edges containing $B$.)
Maps $\omega_1,\omega_2:E(K^{*(1)})\to \Z_2$ are called {\it cohomologous} if $\omega_1-\omega_2$
is the sum of some elementary coboundaries $\delta_{K^{*(1)}}\{A,B\}$.

There is a general position PL map $f:K\to\R$ of a graph $K$ such that $f(A)\not\in f(\sigma)$ for each vertex $A$ and edge $\sigma\not\ni A$ (i.e. $K$ is almost embeddable in the line) if and only if the intersection cocycle of some general position PL map $K\to\R$ is cohomologous to zero.

(f) {\it A cocycle } is a map $E(K^{*(1)})\to\Z_2$ such that the sum of images of edges $A\times CD$, $B\times CD$, $C\times AB$, $D\times AB$ is even for any non-adjacent edges $AB,CD$ of graph $K$, see (b).
Then $\delta_{K^{*(1)}}\{A,B\}$ is a cocycle.

(g) For each cocycle $\nu$ assign to any unordered pair $\{AB,CD\}$ of disjoint edges of graph $K$ the sum of two numbers on `opposite' edges $A\times CD$ and $B\times CD$ of the `square' $AB\times CD$.
That is, define a map $\Sq^1\nu:K^*\to\Z_2$ by the formula
$$
\Sq\phantom{}^1\nu\{AB,CD\}:=\nu(A\times CD)+\nu(B\times CD)=\nu(AB\times C)+\nu(AB\times D).
$$
Then $\Sq^1(\mu+\nu)=\Sq^1\mu+\Sq^1\nu$ and
$\Sq^1\delta_{K^{*(1)}}\{A,B\}=\sum\limits_{\sigma\ni B}\delta(A,\sigma)=:\delta(A\times\delta_K B)$.

(h) Let $H^1(K^*;\Z_2)$ be the group of cohomology classes of cocycles.
Define {\it the van Kampen obstruction} $v_1(K)\in H^1(K^*;\Z_2)$ to $\Z_2$-embeddability of $K$ into the line as the cohomology class of the intersection cocycle of some general position PL map $f:K\to\R$.
This is well-defined analogously to Lemma \ref{star}.
By (g) {\it the Bockstein-Steenrod square} $\Sq^1:H^1(K^*;\Z_2)\to H^2(K^*;\Z_2)$ is well-defined by $\Sq^1[\nu]:=[\nu^2]$.
We have $v(K)=\Sq^1v_1(K)$.

(i) Analogously to (g,h) one can define bilinear {\it Kolmogorov-Alexander product}
\linebreak
$\smile:H^1(K^*;\Z_2)^2\to H^2(K^*;\Z_2)$ for which $\Sq^1 x=x\smile x$.
\end{remark}

\begin{remark}[historical]\label{r:histgr}
In topology (and possibly in other branches of pure mathematics) there is a tradition of unmotivated and artificially complicated exposition (even in textbooks, not only in research papers).
Thus bright results and useful methods of algebraic topology are hard to access to mathematicians working in computer science and in graph theory (see specific remarks below).
Books \cite{BE82, Ma03, Lo13, Sk20, Sk} and this survey attempt to overcome this problem.
(Books \cite{Ma03, Lo13} assume some knowledge of algebraic topology, while books \cite{BE82, Sk20, Sk} do not.)


The Mohar criterion \cite[Theorem 3.1]{Mo89}, cf. \cite[\S2.8]{Sk20} is an easy application of the {\it intersection form on the homology of a manifold} \cite{IF} known in topology since the beginning of the 20th century.

Proof of the `if' part of Proposition~\ref{starpr} (this is $v+w\in X(G)$ from \cite[Theorem 1.18]{Sc13}) is straightforward by the van Kampen finger move \cite{vK32}, fig. \ref{reidall}.V or \cite[Figure 4]{Sc13}.
(So although \cite{vK32} does not have explicit statement for lower dimensions, `Van Kampen ... could only prove the other direction for higher dimensions' is incorrect.)
Hence this result should be attributed to \cite{vK32} as well as to later references mentioned in
\cite[the paragraph after Theorem 1.18]{Sc13} and containing explicit statement for lower dimensions.

Already the correction to van Kampen paper \cite{vK32} mentions orientations
(equivalent to counting intersection points with signs).
The papers \cite{Sh57, Wu58} explicitly count intersection points with signs,
as opposed to `it's first explicitly worked out by Tutte' \cite[Theorem 1.19 and the paragraph before]{Sc13} (referring to the 1970 paper [49]).
Still, Tutte is the more influential reference within graph drawing theory, as opposed to algebraic topology.


Proposition \ref{starz} is attributed in \cite[the paragraph after Theorem 1.19]{Sc13} to Tutte, referring to the 1970 paper [49].
Although I do not know an earlier reference,
note that it was known in topology before 1950
that the van Kampen obstruction is a twisted Euler class of some bundle, and that
analogous characteristic classes have order two, cf. \cite[Problem 11.10]{Pr07}.
This is another instance of the algebraic topology tradition and the graph theory tradition not overlapping, and many results and insights only being passed on orally within the different groups.
\end{remark}

\section{Multiple intersections in combinatorial geometry}\label{s:mucoge}

\subsection{Radon and Tverberg theorems in the plane}\label{s:ratvpl}

\jonly{The reader can find more complete exposition and illustrative examples e.g. in \cite[\S2.1]{Sk18}.}


\begin{theorem}[Radon theorem in the plane] \label{0-radpl}
For any $4$ points in the plane either one of them belongs to the triangle with vertices at the others, or they can be decomposed into two pairs such that the segment joining the points of the first pair intersects the segment joining the points of the second pair.
\end{theorem}

Cf. Proposition \ref{0-ra2}.a and Theorems \ref{ratv-totv}, \ref{t:lr}.

The following simple examples show that this result is `best possible':

$\bullet$ In the plane consider vertices of a triangle and a point inside it.
For any partition of these $4$ points into two pairs the segment joining the points of the first pair does not intersect the segment joining the points of the second pair.

$\bullet$ In the plane consider vertices of a square.
None of these $4$ points belongs to the triangle with vertices at the others.

The {\bf convex hull} $\left<X\right>$ of a finite subset $X\subset\R^2$ is the smallest convex polygon which contains $X$.

Radon theorem in the plane can be reformulated as follows: {\it any $4$ points in the plane can be decomposed into two disjoint sets whose convex hulls intersect.}
This reformulation has the following stronger `quantitative' form.

\begin{proposition}[see proof in \jonly{\cite[\S2.5]{Sk18}} \S\ref{s:apptve};
cf. Proposition \ref{0-ra2}.b and Lemma \ref{ratv-vk2}] \label{ratv-vk2l}
If no $3$ of $4$ points in the plane belong to a line, then there exists a unique partition of these $4$ points into two sets whose convex hulls intersect.
\end{proposition}

Now consider partitions of subsets of the plane into {\it more than two} disjoint sets.

\begin{example}\label{ratv-9ex}
(a) In the plane take a pair of points at each vertex of a triangle (or a `similar' set of distinct points).
For any decomposition of these six points into three disjoint sets the convex hulls of these sets do not have a common point.

(b) In the plane take $r-1$ points at each vertex of a triangle (or a `similar' set of distinct points).
For any decomposition of these $3r-3$ points into $r$ disjoint sets the convex hulls of these sets do not have a common point.

(c) In the plane take vertices of a convex 7-gon.
For any numbering of these $7$ points from 1 to 7, point $1$ does not belong to any of the (2-dimensional) triangles $234$ and $567$.

(d) In the plane take an equilateral triangle $ABC$ and its center $O$.
Define points $A_1, B_1, C_1$ as the images of points $A, B, C$ under homothetic transformation
with the center $O$ and the ratio $1/2$.
For any numbering of these $7$ points from 1 to 7 the intersection point of the segments $12$ and $34$
does not belong to the (2-dimensional) triangle $567$ (see proof in \S\ref{s:apptve}).
\end{example}


\begin{pr}[see proof in \S\ref{s:apptve}\jonly{\cite[\S2.5]{Sk18}}]\label{ratv-9}
(a) The vertices of any convex octagon can be decomposed into three disjoint sets whose  convex hulls have a common point.

(b) Any $11$ points in the plane can be decomposed into three disjoint sets whose  convex hulls have a common point.

(c) For any $r$ there exist $N$ such that any $N$ points in the plane can be decomposed into $r$ disjoint sets
whose convex hulls have a common point.
\end{pr}

By part (b), in part (c) for $r=3$ we can take $N=11$.
By Example~\ref{ratv-9ex}.a every such $N$ is greater than $6$.
It turns out that the minimal $N$ is $7$; this fact is nontrivial.
The following theorem shows that for general $r$ the minimal $N$ is just one above the number of Example~\ref{ratv-9ex}.b.

\begin{figure}[h]\centering
\includegraphics{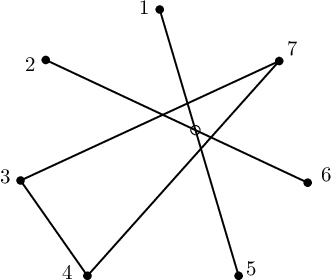}\qquad\qquad \includegraphics{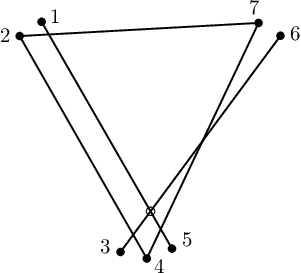}
\caption{A common point of convex hulls as in Theorem~\ref{ratv-tv1}}\label{tverbpoints}
\end{figure}


\begin{theorem}[Tverberg theorem in the plane, cf. Theorem \ref{t:lt}] \label{ratv-tv1}
For any $r$ every $3r-2$ points in the plane can be decomposed into $r$ disjoint sets whose convex hulls have a common point.
\end{theorem}

For a motivated exposition of the well-known proof see \cite{RRS}.

\begin{example}\label{ex:tverbpoints}
(a) For the vertices of regular heptagon the number of partitions from Theorem~\ref{ratv-tv1} is $7$.
Every such partition looks like rotated partition of fig.~\ref{tverbpoints}, left.

(b) For the points in fig.~\ref{tverbpoints}, right, the number of partitions from Theorem~\ref{ratv-tv1} is $4$. This follows because for every such partition one of the convex hulls is a triangle with one vertex $4$, another vertex $1$ or $2$, and the third vertex $6$ or $7$.

(c) (cf. Propositions \ref{0-ra2} and \ref{ratv-vk2l})
The sum
$$\sum\limits_{\{R_1,R_2,R_3\}\ :\ M_i=R_1\sqcup R_2\sqcup R_3} |\left<R_1\right>\cap\left<R_2\right>\cap\left<R_3\right>|$$
has different parity for the two sets $M_1,M_2$ of (a), (b). \jonly{fig.~\ref{tverbpoints}.}
\end{example}


\subsection{Topological Radon theorem in the plane}\label{0-ratvtopl}

\begin{proposition}[see proof in \jonly{{\cite[\S2.5]{Sk18}}} \S\ref{s:apptve}]\label{ratv-chess}
Take a closed polygonal line $L$ in the plane whose vertices are in general position.

(a) The complement to $L$ has a chess-board coloring (so that the adjacent domains have different colors, see fig.~\ref{f:tohu}).


(b) (cf. Proposition \ref{ratv-totv2}.c)
The ends of a polygonal line $P$ whose vertices together with the vertices of $L$ are in general position have the same color if and only if $|P\cap L|$ is even.
\end{proposition}

\begin{figure}[h]\centering
\includegraphics[scale=.8]{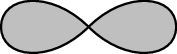} \qquad \includegraphics[scale=.8]{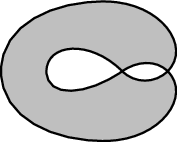} \qquad \includegraphics[scale=.8]{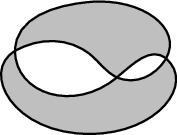} \qquad \includegraphics[scale=.8]{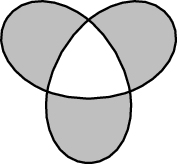} \qquad \includegraphics[scale=.8]{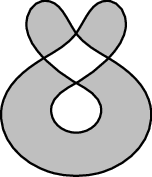}
\caption{The modulo two interiors of some closed polygonal lines}\label{f:tohu}
\end{figure}

The {\bf modulo two interior} of a closed polygonal line in the plane whose vertices are in general position is the union of black domains for a  chess-board coloring (provided the infinite domain is white).

Piecewise-linear (PL) and general position PL maps $K_n\to\R^2$ are defined in \S\ref{0grapl}.

\begin{theorem}[Topological Radon theorem in the plane \cite{BB}, cf. Theorems \ref{grapl-nonalm}, \ref{0-radpl}, \ref{t:tr}] \label{ratv-totv}

(a) For any general position PL map $f:K_4\to\R^2$ either

$\bullet$ the images of some non-adjacent edges intersect, or

$\bullet$ the image of some vertex belongs to the interior modulo 2
of the image of the cycle formed by those three edges that do not contain this vertex.

(b) For any PL (or continuous) map of a tetrahedron to the plane\footnote{See definition e.g. in \cite[\S3]{Sk20}.}
either

$\bullet$ the images of some opposite edges intersect, or

$\bullet$ the image of some vertex belongs to the image of the opposite face.
\end{theorem}


Part (a) follows from its `quantitative version' Lemma~\ref{ratv-vk2} below using a version of
\cite[Approximation Lemma 1.4.6]{Sk20}, cf. Remark \ref{r:ae}.c.


Part (b) for PL general position maps follows from part (a) because the image $f\Delta$ of a face $\Delta$ contains the interior modulo 2 of the image of the boundary $\partial\Delta$ of this face.
(This fact follows because for a {\it general position} map $f:\Delta\to\R^2$ a {\it general position} point from the interior modulo 2 of $f(\partial\Delta)$ has an odd number of $f$-preimages.)
Part (b) follows from part (b) for general position PL maps using a version of \cite[Approximation Lemma 1.4.6]{Sk20}, cf. Remark \ref{r:ae}.c.


Also, the standard formulation (b) is equivalent to (a) by \cite{Sc04, SZ}.

For any general position PL map $f:K_4\to\R^2$ let {\bf the Radon number} $\rho(f)\in\Z_2$ be
the sum of the parities of

$\bullet$ the number of intersections points of the images of non-adjacent edges, and

$\bullet$ the number of vertices whose images belong to the interior modulo 2 of the image of the cycle formed by the three edges not containing the vertex.\footnote{For a general position PL map $g$ of a tetrahedron to the plane one can define the {\it van Kampen number} $v(g)\in\Z_2$ \cite[\S4.2]{Sk16} so that $v(g)=\rho(g|_{K_4})$.}

\begin{lemma}[cf.\ Lemma~\ref{11-vankam}, Proposition \ref{ratv-vk2l} and \cite{Sc04, SZ}]\label{ratv-vk2}
For every general position PL map $f:K_4\to\R^2$ the Radon number $\rho(f)$ is odd.
\end{lemma}

\begin{proof} By Proposition \ref{ratv-vk2l} it suffices to prove that $\rho(f)=\rho(f')$ for each two general position PL maps $f,f':K_4\to\R^2$ coinciding on every edge except an edge $\sigma$, and such that $f|_\sigma$ is linear.
Denote by $\tau$ the edge of $K_4$ non-adjacent to $\sigma$, by $S$ the modulo 2 interior of
$\partial S:=f\sigma\cup f'\sigma$.
Then
$$
\rho(f)-\rho(f')=\left(|\partial S\cap f\tau|+|S\cap f(\partial\tau)|\right)\mod2=0.
$$
Here the second equality follows by Proposition~\ref{ratv-chess}.b.\footnote{There is a direct proof that the van Kampen number of a map $K_5\to\R^2$ equals the Radon number of its restriction to $K_4$ \cite[\S4.2]{Sk16}.
So Lemmas \ref{ratv-vk2} and \ref{11-vankam} are not only proved analogously, but can be deduced from each other.}
\end{proof}

\subsection{Topological Tverberg theorem in the plane}\label{s:ttp}

\subsubsection{Statement}\label{s:tvtopl}

The topological Tverberg theorem in the plane \ref{ratv-tvpl} generalizes both the Tverberg Theorem in the plane~\ref{ratv-tv1} and the Topological Radon Theorem in the plane \ref{ratv-totv}.
For statement we need a definition.
{\bf The winding number} of a closed oriented polygonal line $A_1\ldots A_n$ in the plane around a point $O$ that does not belong to the polygonal line is the following sum of the oriented angles divided by $2\pi$
$$A_1\ldots A_n\cdot O :=
(\angle A_1OA_2+\angle A_2OA_3+\ldots+\angle A_{n-1}OA_n+\angle A_nOA_1)/2\pi$$


\begin{proposition}\label{ratv-totv2}
(a) The winding number of any polygon (without self-intersections and oriented counterclockwise) around
any point in its exterior (interior) is 0 (respectively 1).

(b) The interior modulo 2 (fig.~\ref{f:tohu}) of any closed polygonal line is the set of points for which the winding number is odd.


(c) (cf. Proposition \ref{ratv-chess}.b)
Take a closed and a non-closed oriented polygonal lines $L$ and $P$ in the plane, all whose vertices are in general position.
Let $P_0$ and $P_1$ be the starting point and the endpoint of $P$.
Then $L\cdot P=L\cdot\partial P:=L\cdot P_1-L\cdot P_0$.\footnote{The number $L\cdot P$ is defined in \S\ref{0vkamz}.
\newline
This version of the Stokes theorem shows that the complement to $L$ has a {\it M\"obius-Alexander numbering}, i.e. a `chess-board coloring by integers' (so that the colors of the adjacent domains are different by $\pm1$ depending on the orientations;
the ends of a polygonal line $P$ have the same color if and only if $L\cdot P=0$).
\newline
See more in \cite{Wn}.}
\end{proposition}


\begin{figure}[h]\centering
\includegraphics{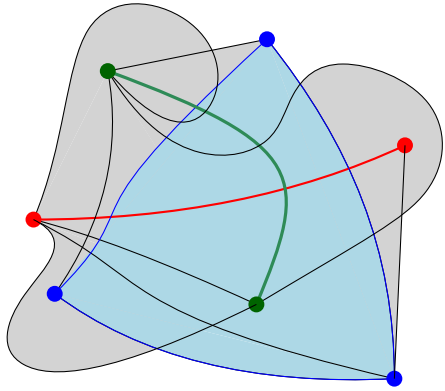}
\caption{Topological Tverberg theorem in the plane, $r=3$}\label{f:lt}
\end{figure}

\begin{theorem}[Topological Tverberg theorem in the plane, \cite{BSS, Oz, vo96}] \label{ratv-tvpl}
If $r$ is a power of a prime, then for any general position PL map $f:K_{3r-2}\to\R^2$
either $r-1$ triangles wind around one vertex or $r-2$ triangles wind around the intersection of two edges, where the triangles, edges and vertices are disjoint.
More precisely, the vertices can be numbered by $1,\ldots,3r-2$ so that either

$\bullet$ the winding number of each of the images $f(3t-1,3t,3t+1)$, $t=2,3,\ldots,r-1$, around some point of $f(12)\cap f(34)$ is nonzero, or

$\bullet$ the winding number of each of the images $f(3t-1,3t,3t+1)$, $t=1,2,3,\ldots,r-1$, around $f(1)$ is nonzero.

(The condition `winding number is nonzero' does not depend on orientation of $f(ijk)$.)
\end{theorem}

Cf. Theorems \ref{ratv-tv1}, \ref{ratv-totv} and Conjecture \ref{c:tt}.

By \cite{Sc04, SZ} Theorem \ref{ratv-tvpl} is equivalent to the following standard formulation:
{\it If $r$ is a power of a prime, then for any PL (or continuous) map of the $(3r-3)$-simplex to the plane there exist $r$ pairwise disjoint faces whose images have a common point.}
Proofs of Theorem \ref{ratv-tvpl} can be found e.g. in \S\ref{s:triple}, \S\ref{s:rainbow} for a prime $r$, and in the surveys cited in `historical notes' of the Introduction for a prime power $r$.


\begin{conjecture}[Topological Tverberg conjecture in the plane]\label{ratv-tvplc}
The analogue of the previous theorem remains correct if $r$ is not a power of a prime.
\end{conjecture}

Let us state a refinement of Theorems~\ref{ratv-tv1} and \ref{ratv-tvpl}.
The following definition is motivated by the refinement and Lemma \ref{cont-homeq}.
An ordered partition $(R_1,R_2,R_3)$ of $M=R_1\sqcup R_2\sqcup R_3\subset[7]$ into three sets (possibly empty) is called {\bf spherical} if
no set $R_1,R_2,R_3$
contains any of the subsets $\{1,2\},\{3,4\},\{5,6\}$.
Generally, an ordered partition $(R_1,\ldots,R_r)$ of $M=R_1\sqcup\ldots\sqcup R_r\subset[3r-2]$ into $r$ sets (possibly empty) is called {\bf spherical} if for every $j=1,\ldots,\lfloor3r-2\rfloor/2$ if $2j-1\in R_s$, then $2j\in R_{s-1}\cup R_{s+1}$, where the $r$ sets are numbered modulo $r$.
Or, less formally, if consecutive odd and even integers are contained in consecutive sets.
Spherical partitions appeared implicitly
in \cite{VZ93} and in \cite[\S6.5, pp. 166-167]{Ma03}.
Cf. \jonly{\cite[Remark 2.5.2]{Sk18}.} Remark \ref{r:rasp}.


\begin{example}\label{ratv-count}
(a) There are $6^3=216$
spherical partitions of $[6]$ into three sets.
Indeed, each of the pairs $\{1, 2\}, \{3, 4\}, \{5, 6\}$ can be distributed in 6 ways.

(b) A spherical partition $(\{1\},\{2,4,5\},\{3\})$ of $[5]$ extends to two spherical partitions $(\{1,6\},\{2,4,5\},\{3\})$ and $(\{1\},\{2,4,5\},\{3,6\})$ of $[6]$.
The extension $(\{1\},\{2,4,5,6\},\{3\})$ is not spherical because $\{2,4,5,6\}\supset\{5,6\}$.
\end{example}

\begin{theorem}\label{ratv-tv7}
(a) For any prime $r$ any $3r-2$ points $1,\ldots,3r-2$ in the plane can be spherically partitioned into $r$ sets 
whose convex hulls have a common point.



(b) For any prime $r$ and general position PL map $f:K_{3r-2}\to\R^2$ either the images of $r-1$ triangles wind around the images of the remaining vertex, or $r-2$ triangles wind around the intersection of two edges.
Moreover, the triangles and the vertex, or the triangles and the edges, respectively, form a spherical partition of $[3r-2]=V(K_{3r-2})$.
\end{theorem}


Part (a) follows from (b).
Part (b) is essentially proved in \cite{VZ93}, \cite[\S6.5]{Ma03}.
This formulation is from \cite[Theorem 3.3.1]{Sc04}, \cite[Theorem 5.8]{SZ}.
(The latter two papers instead of giving a direct proof following \cite[\S6.5]{Ma03}, \cite{VZ93}, see \S\ref{s:rainbow}, deduce Theorem \ref{ratv-tv7}.b in \cite[Proposition 1.6]{SZ} from the 3-dimensional analogue of Theorem \ref{ratv-tv7}.b, which they prove using the same method of \cite[\S6.5]{Ma03}.)


\subsubsection{Ideas of proofs}\label{s:cont-int}

In the following subsubsections we present a standard proof, and an idea of an elementary proof, of the topological Tverberg Theorem for the plane~\ref{ratv-tvpl} (in fact of Theorem \ref{ratv-tv7}).

The idea of an elementary proof presented in \S\ref{s:triple} generalizes the proofs of Lemmas \ref{11-vankam}, \ref{star}, \ref{ratv-vk2} and \cite[Lemmas 6 and 7]{ST17}.
Instead of counting {\it double} intersection points we count {\it $r$-tuple} intersection points.
Instead of counting points {\it modulo 2} we have to count points {\it with signs}, see Example~\ref{ex:tverbpoints}.c.
It is also more convenient (because of Lemma \ref{cont-homeq})
instead of summing over {\it all} partitions to sum over {\it spherical} partitions.
This is formalized by Problem \ref{p:boundary} below which is a `quantitative version' of Theorems~\ref{ratv-tvpl} and~\ref{ratv-tv7}.
Formally, Theorems~\ref{ratv-tvpl} and~\ref{ratv-tv7} for $r=3$ follow by resolution of Problem \ref{p:boundary}.
Proofs of Theorems~\ref{ratv-tvpl},~\ref{ratv-tv7} and of Conjecture \ref{c:tt} for (arbitrary $d$ and) prime $r$ would perhaps be analogous.

Similar proofs of Theorems~\ref{ratv-tvpl}, \ref{ratv-tv7}, and of Conjecture \ref{c:tt} for prime $r$,
are given in \cite{BMZ15, MTW12}, see \jonly{\cite[Remark 2.5.2]{Sk18}} Remark \ref{r:rasp}.
Those papers use more complicated language not necessary for these results (Sarkaria-Onn transform in \cite{MTW12},
homology and equivariant maps between configuration spaces in \cite{BMZ15}).


The standard proof of Theorems~\ref{ratv-tvpl} is presented in \S\ref{s:rainbow} following \cite{VZ93}, \cite[\S6]{Ma03}, cf. \cite{BSS}, \cite[\S2]{Sk16}.
This proof also yields Theorem~\ref{ratv-tv7} and generalizes to Conjecture \ref{c:tt} for prime $r$.
Theorem~\ref{ratv-tvpl} is deduced from the $r$-fold Borsuk-Ulam Theorem \ref{t:3bu} using Lemma \ref{cont-homeq}, both below.
Proof of the Borsuk-Ulam Theorem \ref{ratvto-bu} via its `quantitative version' Lemma \ref{l:buqu}
generalizes to a proof of the $r$-fold Borsuk-Ulam Theorem \ref{t:3bu}.
So this deduction of Theorem~\ref{ratv-tvpl} is not a proof essentially different from the idea of \S\ref{s:triple} but rather the same proof in a different language.
(Therefore it is not quite correct that the main idea of the proof of the topological Tverberg Theorem is to apply the $r$-fold Borsuk-Ulam Theorem for configuration spaces.)

The case $r=3$ gives a non-trivial generalization of the case $r=2$; the generalization to arbitrary $r$ (prime for some results below) is trivial.

\subsubsection{Triple self-intersection invariant for graph drawings}\label{s:triple}


Suppose that every object of $P_1,\ldots,P_r$ is either a point, or an oriented non-closed polygonal line,
or an oriented closed polygonal line, in the plane, all of whose vertices are in general position.
Define {\bf the $r$-tuple algebraic intersection number} $P_1\cdot\ldots\cdot P_r$ to be

(A) $\sum\limits_{X\in P_i\cap P_j} \sgn X \left(\prod\limits_{s\ne i,j} (P_s\cdot X)\right)$, if $P_i,P_j$ are non-closed polygonal lines for some $i<j$,
and the other $P_s$ are closed polygonal lines;

(B) $\prod\limits_{s\ne i} (P_s\cdot P_i)$, if $P_i$ is a point and the other $P_s$ are closed polygonal lines.

Here $\sgn X$ and $\cdot$ are defined in \S\S \ref{0thint}, \ref{0vkamz} and \ref{s:tvtopl};
the number $P_1\cdot\ldots\cdot P_r$ is only defined in cases (A) and (B).\footnote{This is an elementary interpretation in the spirit of \cite{Sc04, SZ} of the $r$-tuple algebraic intersection number $fD^{n_1}\cdot\ldots\cdot fD^{n_r}$ of a general position map $f:D^{n_1}\sqcup\ldots\sqcup D^{n_r}\to\R^2$, where $n_1,\ldots,n_r\subset\{0,1,2\}$ and $n_1+\ldots+n_r=2r-2$ \cite[\S2.2]{MW15}.
This agrees with \cite[\S2.2]{MW15} by \cite[Lemma 27.b]{MW15}.
For a degree interpretation see Assertion \ref{a:3redmap}.\jonly{\cite[Assertion 2.5.4]{Sk18}.}}


\begin{example}\label{e:gsgn}
Assume that $R_1,R_2,R_3$ are either

(A) two vectors and an oriented triangle, or \quad (B) two oriented triangles and a point,

in the plane.
Assume that the vertices of $R_1,R_2,R_3$ are pairwise disjoint subsets of the plane and their union is in general position.
Then $|R_1\cdot R_2 \cdot R_3|=\left<R_1\right>\cap \left<R_2\right>\cap\left<R_3\right>$.

When the intersection is non-empty, $R_1\cdot R_2\cdot R_3=+1$ if and only if up to a  permutation of $(R_1,R_2,R_3)$ not switching the order of vectors, the triangle $R_1$ has the same orientation as the triangle

(A) $A_2B_2B_3$, where $R_i=\overrightarrow{A_iB_i}$ for each $i=2,3$, \quad (B) $R_2$.
\end{example}

\begin{problem}\label{p:boundary}
Let $\mathrm{S}$ be the set of all spherical partitions $(T_1,T_2,T_3)$ of $[7]$ such that $7\in T_3$.
Define a map $\sgn:\mathrm{S}\to\{+1,-1\}$ so that for any general position PL map $f:K_7\to\R^2$
the following (`triple van Kampen') number is not divisible by 3:
$$
V(f):=\sum\limits_{T=(T_1,T_2,T_3)\in \mathrm{S}} \sgn T(fT_1\cdot fT_2\cdot fT_3).
$$
Here $fT_s$ is the $f$-image of either a vertex $T_s$, or an oriented edge $T_s=ab$, $a<b$, or an oriented cycle $T_s=abc$, $a<b<c$.
\end{problem}

Analogously to Lemmas~\ref{11-vankam}, \ref{star}, \ref{ratv-vk2}  and to \cite[Lemmas 6 and 7]{ST17},
the non-divisibility in Problem \ref{p:boundary} could possibly be proved
by calculating $V(f)$ for a specific $f$ and showing that $V(f)$ modulo 3 is independent of $f$.
This might be not so easy, cf. \cite[second half of \S8]{MTW12}.

\subsubsection{An approach via Borsuk-Ulam theorem}\label{s:rainbow}

A map $f:S^n\to\R^m$ is called {\it odd}, or {\it equivariant}, or {\it antipodal} if $f(-x)=-f(x)$ for any $x\in S^n$.
We consider only continuous maps and omit `continuous'.

\begin{theorem}[Borsuk-Ulam]\label{ratvto-bu}
(a) For any map $f:S^d\to\R^d$ there exists $x\in S^d$ such that $f(x)=f(-x)$.

\textup{(a')}~For any equivariant maps $f\colon S^d\to\R^d$ there exists $x\in S^d$ such that $f(x)=0$.

\textup{(b)}~There are no equivariant maps $S^d\to S^{d-1}$.

\textup{(b')}~No equivariant map $S^{d-1}\to S^{d-1}$ extends to $D^d$.

\textup{(c)}~If $S^d$ is the union of $d+1$ closed sets (or $d+1$ open sets), then one of the sets contains opposite points.
\end{theorem}

For $n=2$ part (a) means that at any moment there are two antipodal points on the Earth at
which the temperature and the pressure coincide.


The equivalence of these assertions is simple.
Part (a') is deduced from its following `quantitative version'.

\begin{lemma}\label{l:buqu} If  $0\in\R^d$ is a regular point of a (PL or smooth) equivariant map $f:S^d\to\R^d$, then $|f^{-1}(0)|\equiv2\mod4$.
\end{lemma}

See the definition of a regular point e.g. in \cite[\S8.3]{Sk20}.
Proof of Lemma \ref{l:buqu} is analogous to Lemmas~\ref{11-vankam} and \ref{star} (cf. Problem \ref{p:boundary}):
calculate $|f^{-1}(0)|$  for a specific $f$ and prove that $|f^{-1}(0)|$ modulo 4 is independent of $f$.
Realization of this simple idea is technical, see \cite[\S2.2]{Ma03}.
For other proofs of Theorem \ref{ratvto-bu} see \cite{Ma03}
and the references therein.

Assume that subsets $U,V\subset\R^d$
lie in skew affine subspaces.
For $U,V\ne\emptyset$ define the (geometric) {\bf join}
$$U*V:=\{tx+(1-t)y\in\R^d\ :\ x\in U,\ y\in V,\ t\in[0,1]\}.$$
Define $U*\emptyset:=U$ and $\emptyset*V:=V$. 

The {\it topological join} $U*V$ is the topological space obtained from $U\times V\times[0,1]$ by identifications $(x,y,0)\sim(x',y,0)$ and $(x,y,1)\sim(x,y',1)$ for each $x,x'\in U$, $y,y'\in V$. 
(If you do not know the quotient construction for topological spaces, then regard this as an informal interpretation.)


\begin{pr}\label{p:ext}
If $U$ and $V$ are unions of faces of some simplex $\Delta_n$
and are disjoint, then $U*V$ is the union of all faces of $\Delta_n$ that correspond to subsets
$\sigma\sqcup\tau$, where $\sigma,\tau\subset[n+1]$ correspond to faces of $U,V$, respectively.
\end{pr}

Complexes $K$ and $L$ are called {\it isomorphic} if there is a bijection~$f\colon V(K)\to V(L)$ such that \textit{a subset $A\subset V(K)$ is a face if and only if the subset $f(A)\subset V(L)$ is a face}. 
Notation: $K\sim L$.


For complexes $K$ and $L$ take a complex $L'$ isomorphic to $L$ such that $V(K)\cap V(L')=\emptyset$.
Then the (simplicial) {\bf join} $K*L$ is the complex with the set $V(K)\sqcup V(L')$ of vertices and the set
$\{\sigma\sqcup\tau\ :\ \sigma\in F(K),\ \tau\in F(L')\}$ of faces.
Clearly,
$$\con K\sim K*[1],\quad K=K*\emptyset\subset K*L,\quad K*L\sim L*K,\quad\text{and}\quad
(K_1*K_2)*K_3\sim K_1*(K_2*K_3).$$
\algor{For subsets $U_1,\ldots,U_r\subset\R^d$ lying in skew affine subspaces the (geometric) $r$-tuple join is
$$U_1*\ldots *U_r:=\{\ t_1x_1+\ldots+t_rx_r\in\R^d\ :\ x_j\in U_j,\ t_j\in[0,1],\ t_1+\ldots+t_r=1\ \}.$$
Analogously one interprets the (simplicial) $r$-tuple join of $r$ complexes.
The $k$-tuple join of $k$ copies of a complex $K$ is denoted by $K^{*k}$.


\begin{pr}\label{joinass}
(a) $[1]^{*k}\sim D^{k-1}$; \quad (b) $D^k*D^l\sim D^{k+l+1}$.
\end{pr}

}

For more discussions of the geometric, topological, and combinatorial join see \cite[\S4.2]{Ma03}.

Assertion \ref{p:ext} implies that to every ordered partition of $[3r-3]$ into $r$ sets there corresponds a $(3r-4)$-simplex of $\Delta_{3r-4}^{*r}:=\Delta_{3r-4}*\ldots*\Delta_{3r-4}$ ($r$ `factors').
Denote by $|\mathrm{S}_r|$ the union of $(3r-4)$-simplices of $\Delta_{3r-4}^{*r}$
corresponding to spherical partitions of $[3r-3]$ into $r$ sets.


\begin{pr}\label{cont-hom}
The union $|\mathrm{S}_r|$ is PL homeomorphic to $S^{3r-4}$.
\end{pr}

This assertion and the following lemma are easily deduced from $(S^1)^{*r}\cong S^{2r-1}$ \cite[\S4.2]{Ma03}, \cite[\S5.5 `PL homeomorphism of complexes']{Sk}, see details in \cite[pp. 166-167]{Ma03}.

Denote by $\Sigma_r$ the permutation group of $r$ elements.
The group $\Sigma_r$ acts on the set of real $3\times r$-matrices by permuting the columns.
Denote by $S^{3r-4}_{\Sigma_r}$ the set formed by all those of such matrices, for which the sum in every row is zero, and the sum of squares of the matrix elements is 1.
This set is homeomorphic to the sphere of dimension $3r-4$.
Take a triangulation of this set given by some such homeomorphism.
This set is invariant under the action of $\Sigma_r$.
The cyclic permutation $\omega:S^{3r-4}_{\Sigma_r}\to S^{3r-4}_{\Sigma_r}$ of the $r$ columns has no fixed points and $\omega^r=\id S^{3r-4}_{\Sigma_r}$.

\begin{lemma}\label{cont-homeq} There is a PL homeomorphism $h:|\mathrm{S}_r|\to S^{3r-4}_{\Sigma_r}$ such that $h(R_2,\ldots,R_r,R_1)$ is obtained from $h(R_1,R_2,\ldots,R_r)$ by cyclic permutation of the $r$ columns.
\end{lemma}

\begin{theorem}[$r$-fold Borsuk-Ulam Theorem]\label{t:3bu} Let $r$ be a prime and $\omega:S^k\to S^k$ a PL map without fixed points such that $\omega^r=\id S^k$.
Then no map $g:S^k\to S^k$ commuting with $\omega$ (i.e.  such that $g\circ\omega=\omega\circ g$) extends to $D^{k+1}$.
\end{theorem}

\begin{proof}[Comments on the proof]
Clearly, the theorem is equivalent to the following result.

{\it Extend $\omega$ to $S^k*\Z_3$ by $\omega(ts\oplus(1-t)m):=t\omega(s)\oplus(1-t)(m+1)$.
Let $\omega_0:\R^{k+1}\to \R^{k+1}$ be a map whose only fixed point is 0 and such that $\omega_0^r=\id \R^{k+1}$.
Then for any map $g:S^k*\Z_3\to \R^{k+1}$ commuting with $\omega,\omega_0$ (i.e. such that $g\circ\omega=\omega_0\circ g$)
there is $x\in S^k*\Z_3$ such that $g(x)=0$.}

This result is deduced from its `quantitative version' analogous to Lemma \ref{l:buqu}.

For a standard proof see \cite{BSS}, \cite[\S6]{Ma03}.
\end{proof}

\begin{proof}[Proof of Theorem~\ref{ratv-tvpl}]
Consider the case $r=3$, the general case is analogous.
We use the standard for\-mu\-la\-tion of Theorem~\ref{ratv-tvpl} given after the statement.
Suppose to the contrary that $f:\Delta_6\to\R^2$ is a continuous map and there are no 3 pairwise disjoint faces whose images have a common point.

For $x\in \R^2$ let $x^*:=(1,x)\in\R^3$.
For $x_1,x_2,x_3\in\R^2$ and $t_1,t_2,t_3\in[0,1]$ such that $t_1+t_2+t_3=1$ and pairs $(x_1,t_1),(x_2,t_2),(x_3,t_3)$ are not all equal define
$$
S^*:=t_1x_1^*+t_2x_2^*+t_3x_3^*,\quad \pi^{*'}:=
\left(t_1x_1^*-\frac{S^*}3,t_2x_2^*-\frac{S^*}3,t_3x_3^*-\frac{S^*}3\right)
\quad\text{and}\quad \pi^*:=\frac{\pi^{*'}}{|\pi^{*'}|}.
$$
This defines a map
$$
\pi^*:\R^2*\R^2*\R^2-\diag\phantom{}^*\to S^5_{\Sigma_3},\quad\text{where}\quad
\diag\phantom{}^*:=\left\{\left(\frac13x\oplus\frac13x\oplus\frac13x\right)\right\}.
$$
So we obtain the map $\pi^*\circ(f*f*f):|\mathrm{S}_3|\to S^5_{\Sigma_3}$.
This map extends to the union of 6-simplices of $\Delta_6^{*3}$ corresponding to spherical partitions $(T_1,T_2,T_3)$ of $[7]$ into 3 sets such that $7\in T_3$.
The union is PL homeomorphic to $\con|\mathrm{S}_3|\cong D^6$.
The map $\pi^*\circ(f*f*f)$ commutes with the cyclic permutations of the three sets in $|\mathrm{S}_3|$ and of the three columns in $S^5_{\Sigma_3}$.
Take any PL homeomorphism $h$ of Lemma \ref{cont-homeq}.
The composition $g:=\pi^*\circ(f*f*f)\circ h^{-1}:S^5_{\Sigma_3}\to S^5_{\Sigma_3}$ commutes with the cyclic permutation of the three columns and extends to $D^6$.
A contradiction to Theorem \ref{t:3bu}.\footnote{Here instead of using Lemma \ref{cont-homeq} we can use that the map $\pi^*\circ(f*f*f):(\Delta_6)^{*3}_{\Delta}\to S^5_{\Sigma_3}$ is well-defined on {\it the triple deleted join} $(\Delta_6)^{*3}_{\Delta}\cong (\Z_3)^{*7}$, prove that $(\Z_3)^{*7}$ is 5-connected, and then construct a $\Z_3$-equivariant map $S^5_{\Sigma_3}*\Z_3\to (\Delta_6)^{*3}_{\Delta}$.}
\end{proof}



\subsection{Mapping complexes in the plane and the \"Ozaydin theorem}\label{s:tvkam}

Formulation of the \"Ozaydin Theorem \ref{vankamz-oz} uses the definition of a {\it multiple ($r$-fold) intersection cocycle}.
We preface the definition by simplified analogues.
In \S\ref{0vkam2} we have defined double ($2$-fold) intersection cocycle for graphs.
In \S\ref{s:radc2} we define double intersection cocycle modulo 2 for complexes (or hypergraphs).
In \S\ref{s:radcz} we generalize that definition from residues modulo 2 to integers.
In \S\ref{s:tvescz} we generalize definition of \S\ref{s:radcz} from $r=2$ to arbitrary $r$.

\subsubsection{A polynomial algorithm for recognizing planarity of complexes}\label{s:radc2}

\UseRawInputEncoding

A {\bf $k$-hypergraph} (more precisely, $k$-dimensional, or $(k+1)$-uniform, hypergraph) $(V,F)$ is a finite set $V$ together with a collection $F\subset{V\choose k+1}$ of $(k+1)$-element subsets of $V$.

In topology it is more traditional (because sometimes more convenient) to work not with hypergraphs but with
{\it complexes} (we shall not use longer name `abstract finite simplicial complexes').
The following results are stated for complexes, although some of them are correct for hypergraphs.

A {\bf complex} $K=(V,F)$ is a finite set $V=V(K)$ together with a collection $F=F(K)\subset 2^V$ of subsets of $V$ such that if a subset $\sigma$ is in the collection, then each subset of $\sigma$ is in the collection.
(Hence $F\ni\emptyset$.)
In an equivalent geometric language, a complex is a collection of closed faces (=subsimplices) of some simplex.
A {\bf $k$-complex} is a complex containing at most $(k+1)$-element subsets, i.e. at most $k$-dimensional simplices.

Elements of $V$ and of $F$ are called {\bf vertices} and {\bf faces}.
\algor{{\it Ребром}  называется двухэлементная (т.е. одномерная) грань.}
\invadraw{An {\it edge} is a 2-element (t.e., 1-dimensional) face.}


{\bf The complete $k$-complex on $n$ vertices} (or the $k$-skeleton of the $(n-1)$-simplex)
$\Delta_{n-1}^k:=([n],{[n]\choose \le k+1})$ is the collection of all at most $(k+1)$-element subsets of an $n$-element set.
For $k=0$ we denote this complex by $[n]$, for $n=k+1$ by $D^k$ ($k$-simplex or $k$-disk), and for $n=k+2$ by $S^k$ ($k$-sphere).








Let $K=(V,F)$ be 2-complex.
The graph $(V,E)$ formed by vertices and edges of $K$ is denoted by $K^{(1)}$.
E.g. $(\Delta_n^2)^{(1)}=K_{n+1}$.
For a PL map $f:K^{(1)}\to\R^2$ and a face $R=\{A,B,C\}$ the closed polygonal line $f(AB)\cup f(BC)\cup f(CA)$ is denoted by $f(\partial R)$.

A 2-complex $K$ is called {\bf planar} (or PL embeddable into the plane) if there exists a PL embedding $f:K^{(1)}\to\R^2$ such that for any vertex $A$ and face $R\not\ni A$ the image $fA$ does not lie inside the polygon $f(\partial R)$.
Cf. \cite[Definition 3.0.6 for $q=2$]{Sc04}.

For illustration observe that the complete 2-complex $\Delta_3^2$ on $4$ vertices (i.e. the boundary of a tetrahedron) is not planar (and not even $\Z_2$-planar, see below) by the Topological Radon Theorem in the plane \ref{ratv-totv}.\footnote{More generally, any 2-complex that is the boundary of a convex polyhedron in $\R^3$ is not $\Z_2$-planar \cite{LS98}.
Even more generally, if every edge of a 2-complex $K$ is contained in a positive even number of faces, then $K$ is not $\Z_2$-planar.
This can be proved analogously to Theorem \ref{ratv-totv} or using the Halin--Jung criterion \cite[Appendix A]{MTW}.}


\begin{theorem}[\cite{GR79}; cf. Proposition \ref{grapl-ea} and Theorem \ref{t:rec}]\label{grapl-ea-r}
There is a polynomial algorithm for recognition planarity of 2-complexes.
\end{theorem}

In \cite[Appendix A]{MTW} it is explained that this result (even with linear algorithm) follows from the Kuratowski-type Halin-Jung planarity criterion for 2-complexes (stated there).
We present a different proof similar to proof of Proposition~\ref{grapl-ea}.b (\S\ref{0vkam2}).
This proof illustrates the idea required for elementary formulation of the \"Ozaydin Theorem \ref{vankamz-oz}.

A 2-complex $K$ is called {\bf $\Z_2$-planar} if there exists a general position PL map $f:K^{(1)}\to\R^2$ such that the images of any two non-adjacent edges
intersect at an even number of points and for any vertex $A$ and face $R\not\ni A$ the image $fA$ does not lie in the interior modulo 2 of $f(\partial R)$.

\begin{theorem}[cf. Theorems \ref{vkam-z2} and \ref{t:mawa}]\label{vkam-z2-r}
A 2-complex is planar if and only if it is $\Z_2$-planar.
\end{theorem}

This is proved using the Halin--Jung criterion \cite[Appendix A]{MTW}.


Let $K=(V,F)$ be a 2-complex and $f:K^{(1)}\to\R^2$ a general position PL map.
Assign to any pair $\{\sigma,\tau\}$ of non-adjacent edges the residue
$$|f\sigma\cap f\tau|\mod2.$$
Assign to any pair of a vertex $A$ and a face $R\not\ni A$ the residue
$$|fA\cap\mathrm{int}_2f(\partial R)|\mod2,$$
where $\mathrm{int}_2f(\partial R)$ is the interior modulo 2 of $f(\partial R)$.

Denote by $K^*$ the set of unordered pairs $\{R_1,R_2\}$ of disjoint subsets $R_1,R_2\in V\sqcup E\sqcup F$ such that $|R_1|+|R_2|=4$, where $E$ is the set of edges.
Then either both $R_1$ and $R_2$ are edges, or one of $R_1,R_2$ is a face and the other one is a vertex.
The obtained map $K^*\to\Z_2$ is called {\bf the} (double) {\bf intersection cocycle} (modulo 2) of $f$ for $K$.
Note that $K^*\supset (K^{(1)})^*$ and the intersection cocycle of $f$ for $K^{(1)}$ is the restriction of the intersection cocycle of $f$ for $K$.
The intersection cocycle for $K$ of the map $f:K^{(1)}\to\R^2$ from Example \ref{vkam2-line} is the extension to $K^*$ by zeroes of the intersection cocycle for $K^{(1)}$ described there.

Comparing the definitions of the Radon number and the intersection cocycle we see that for every general position PL map $f:K_4=(\Delta_3^2)^{(1)}\to\R^2$ the Radon number $\rho(f)$ equals to the sum of the values of the intersection cocycle for $\Delta_3^2$.

By Proposition \ref{ratv-totv2}.c for any disjoint edge $\sigma$ and face $R$ we have
$$
\sum\limits_{A\in\sigma}|fA\cap \mathrm{int}_2f(\partial R)| =
\sum\limits_{\tau\subset R}|f\tau\cap f\sigma|.
$$
Analogue of Proposition \ref{elcob} is true for the intersection cocycle for 2-complex, with the following definition.
Let $K$ be a 2-complex and $A$ its vertex which is not the end of an edge $\sigma$.
{\bf An elementary coboundary} of the pair $(A,\sigma)$ is the map $\delta_K(A,\sigma):K^*\to\Z_2$  that assigns $1$ to the pair $\{R_1,R_2\}$ if $R_i\supset A$ and $R_j\supset\sigma$ for some $i\ne j$, and $0$ to any other pair.

The  subset of $\delta_K(A,\sigma)^{-1}(1)\subset K^*$ corresponding to the map $\delta_K(A,\sigma)$ is also called elementary coboundary.
So $\delta_{\Delta_3^2}(1,23)=\big\{\{14,23\},\{1,234\}\big\}$,
cf. Example \ref{cobo}.a.

\begin{proposition}\label{elcob-r} Under the Reidemeister move in fig.~\ref{reidall}.V
(or a move in fig.~\ref{f:11-vankam}) the intersection cocycle changes by adding $\delta_K(A,\sigma)$.
\end{proposition}


Two maps $\nu_1,\nu_2:K^*\to\Z_2$ are called {\bf cohomologous} if
$$
\nu_1-\nu_2=\delta_K(A_1,\sigma_1)+\ldots+\delta_K(A_k,\sigma_k)
$$
for some vertices $A_1,\ldots,A_k$ and edges $\sigma_1,\ldots,\sigma_k$ (not necessarily distinct).

\begin{lemma}[cf. Lemmas \ref{star} and \ref{ratv-vk2}]\label{star-r}
For any 2-complex $K$ the intersection cocycles of different general position PL maps $K^{(1)}\to\R^2$ are cohomologous.
\end{lemma}


\begin{proposition}[cf. Proposition \ref{starpr}]\label{starpr-r}
A 2-complex $K$ is $\Z_2$-planar if and only if the intersection cocycle modulo 2 of some
(or, equivalently, of any) general position PL map $K^{(1)}\to\R^2$ is cohomologous to the zero map.
\end{proposition}

This proposition follows by Lemma~\ref{star-r} and Proposition~\ref{elcob-r}.


\begin{proof}[Proof of Theorem \ref{grapl-ea-r}]
Take a 2-complex $K$.
To every pair $A,\sigma$ of a vertex and an edge such that $A\notin\sigma$ assign a variable $x_{A,\sigma}$.
For every $\{R,R'\}\in K^*$ denote by $b_{R,R'}\in\Z_2$ the value of the extension to $K^*$ by zeroes of the intersection cocycle for $K^{(1)}$ described in Example~\ref{vkam2-line}.
For every such pairs $(A,\sigma)$ and $\{R,R'\}$ let
$$a_{A,\sigma,R,R'}=
\begin{cases}
1 & \text{either ($R\ni A$ and $R'\supset\sigma$) or ($R'\ni A$ and $R\supset\sigma$)}\\
0 & \text{otherwise}
\end{cases}.$$
For every pair $\{R,R'\}\in K^*$ consider the linear equation
$\sum_{A\notin\sigma} a_{A,\sigma,R,R'}x_{A,\sigma}=b_{R,R'}$ over $\Z_2$.
By Theorem \ref{vkam-z2-r} and Proposition \ref{starpr-r} planarity of $K$ is equivalent to solvability of this system of equations.
This can be checked in polynomial time, see details in \cite{CLR, Vi02}.
\end{proof}



\subsubsection{Intersections with signs for $2$-complexes}\label{s:radcz}

Let $K$ be a 2-complex and $f:K^{(1)}\to\R^2$ a general position PL map.
Orient the edges and faces of $K$, i.e. choose some cyclic orderings on the subsets that are edges and faces.
Assign to every ordered pair

$\bullet$ $(\sigma,\tau)$ of non-adjacent edges the algebraic intersection number $f\sigma\cdot f\tau$ (defined in \S\ref{0vkamz}).

$\bullet$ $(A,R)$ or $(R,A)$ of a vertex $A$ and a face $R\not\ni A$ minus the winding number  $-fA\cdot f(\partial R)$ of $f(\partial R)$ around $A$ (defined in \S\ref{s:tvtopl}).

Denote by $\t K$ the set of ordered pairs $(R_1,R_2)$ of disjoint subsets $R_1,R_2\in V\sqcup E\sqcup F$
such that $|R_1|+|R_2|=4$.
The obtained map $\cdot:\t K\to\Z$ is called {\bf the integral intersection cocycle} of $f$ for $K$ (and for given orientations).

For oriented 2-element set $AB$ denote $[AB:B]=1$ and $[AB:A]=-1$.
($AB$ \emph{goes} to $B$ and \emph{issues} out of $A$.)
For oriented 3-element set $ABC$ denote $[ABC:BA]=[ABC:CB]=[ABC:AC]=1$ and $[ABC:AB]=[ABC:BC]=[ABC:CA]=-1$.
For other oriented sets $R,R'\in V\sqcup E\sqcup F$ define $[R:R']=0$.


By Proposition \ref{ratv-chess}.b for any disjoint edge $\sigma$ and face $R$ we have
$$
\sum\limits_{A\in\sigma}[\sigma:A](f(\partial R)\cdot fA) =
\sum\limits_{\tau\subset R}[R:\tau](f\sigma\cdot f\tau).
$$
Analogue of Proposition \ref{elcob} is true for the integral intersection cocycle for 2-complex, with the following definition.
Let $K$ be a 2-complex whose edges and faces are oriented, and $A$ a vertex which is not the end of an edge $\sigma$.
{\bf An elementary super-symmetric coboundary} of the pair $(A,\sigma)$ is the map $\delta_K(A,\sigma):\t K\to\Z$ that assigns
$$-[\tau:A]\text{ to }(\sigma,\tau),\quad [\tau:A]\text{ to }(\tau,\sigma)\quad\text{and}\quad [R:\sigma]\text{ both to }(A,R)\text{ and }(R,A).$$
In other words,
$$\delta_K(A,\sigma)(R_1,R_2) := [R_1:A][R_2:\sigma]+(-1)^{(|R_1|-1)(|R_2|-1)}[R_2:A][R_1:\sigma].$$




\begin{proposition}\label{elcobz-r} Under the Reidemeister move in fig.~\ref{reidall}.V
the integer intersection cocycle changes by adding $\delta_K(A,\sigma)$.
\end{proposition}

Maps $\nu_1,\nu_2:\t K\to\Z$ are called {\bf  super-symmetrically cohomologous} if
$$
\nu_1-\nu_2=m_1\delta_K(A_1,\sigma_1)+\ldots+m_k\delta_K(A_k,\sigma_k)
$$
for some vertices $A_1,\ldots,A_k$, edges $\sigma_1,\ldots,\sigma_k$ and integer numbers $m_1,\ldots,m_k$ (not necessarily distinct).

The integral analogues of Lemma \ref{star-r} and Proposition \ref{starpr-r} are correct, cf. Lemma \ref{star-zl}.


\begin{proposition}[cf. Propositions \ref{starz}, \ref{vankamz-oz3}]\label{starz-r}
For any 2-complex $K$ twice the integral intersection cocycle of any general position PL map $K^{(1)}\to\R^2$ is  super-symmetrically cohomologous to the zero map.
\end{proposition}

This follows by the integral analogue of Lemma \ref{star-r} and the analogue of Assertion \ref{starzch}.c for 2-complexes.

\subsubsection{Elementary formulation of the \"Ozaydin theorem}\label{s:tvescz}

Let $K=(V,F)$ be a 2-complex and $f:K^{(1)}\to\R^2$ a general position PL map.
Denote by $E$ the set of edges.
Orient the edges and faces of $K$, i.e. choose some cyclic orderings on the subsets that are edges and faces.
Denote by $K^{\underline{r}}$ the set of ordered $r$-tuples $(R_1,\ldots,R_r)$ of pairwise disjoint sets from $V\sqcup E\sqcup F$ such that either

(A) two of the sets $R_1,\ldots,R_r$ are edges and the other are faces, or

(B) one of the sets $R_1,\ldots,R_r$ is a vertex and the other are faces.\footnote{This is the $d(r-1)$-skeleton of the {\it simplicial $r$-fold deleted product} of $K$.
Cf. \cite[\S1.4]{Sk16}.}

Clearly,

$\bullet$  if $|V|<3r-2$, then $K^{\underline{r}}=\emptyset$.

$\bullet$ $(\Delta^2_{3r-3})^{\underline{r}}$ is the set of ordered
partitions of $[3r-2]$ into $r$ non-empty subsets, every subset having at most 3 elements.

{\bf The $r$-fold intersection cocycle} of $f$ for $K$ (and for given orientations) is a map $K^{\underline{r}}\to\Z$ that assigns to $r$-tuple $(R_1,\ldots,R_r)$ the number $fR_1\cdot\ldots\cdot fR_r$ or $-fR_1\cdot\ldots\cdot fR_r$ in cases (A) or (B) above, respectively.\footnote{This agrees up to sign with the definition of \cite[Lemma 41.b]{MW15} because by
\cite[(13) in p. 17]{MW15} $\varepsilon_{2,2,\ldots,2,0}$ is even and $\varepsilon_{2,2,\ldots,2,1,1}$ is odd.
\newline
The $r$-fold intersection cocycle depends on an arbitrary choice of orientations, but the triviality condition defined below does not.}


{\it Super-symmetric} $r$-fold {\it elementary coboundary} and {\it cohomology} are defined analogously to the case $r=2$  considered in \S\ref{s:radcz}.
The $r$-fold analogue of Lemma \ref{star-r} is correct with a similar proof.

\begin{remark}\label{r:anar-starpr-r}
It would be interesting to know if the $r$-fold analogue of Proposition \ref{starpr-r} is correct.
\end{remark}

A map $K^{\underline{r}}\to\Z$ is called (super-symmetrically cohomologically) {\bf trivial} if it is super-symmetrically cohomologous to the zero map.

Proofs of the Topological Tverberg Theorem in the plane \ref{ratv-tvpl} (mentioned after the statement) show that {\it if $r$ is a prime power, then for the complete 2-complex $K=\Delta^2_{3r-3}$ the $r$-fold intersection cocycle of every general position PL map $K^{(1)}\to\R^2$ is non-trivial} .

\begin{remark}\label{r:p-boundary} The number from Problem~\ref{p:boundary}
is the sum of some values of the threefold intersection cocycle (with certain coefficients).
\end{remark}

\begin{theorem}[\"Ozaydin]\label{vankamz-oz}
If $r$ is not a prime power, then for every 2-complex $K$ the $r$-fold intersection cocycle of any general position PL map $K^{(1)}\to\R^2$ is trivial.
\end{theorem}

This is implied by the following Proposition~\ref{vankamz-oz3}.b because when $r$ is not a prime power, the numbers $r!/p^{\alpha_{r,p}}$, for all primes $p<r$, have no common multiple.
Here $\alpha_{r,p}=\sum_{k=1}^\infty \left\lfloor \dfrac{r}{p^k}\right\rfloor$ is the power of $p$ in the prime factorisation of $r!$.

\begin{proposition}[cf. Proposition \ref{starz-r}]\label{vankamz-oz3}
Let $K$ be a 2-complex and $f:K^{(1)}\to\R^2$ a general position PL map.

(a) Threefold intersection cocycle of $f$ multiplied by $3$ is trivial.

(b) If $r$ is not a power of a prime $p$, then the $r$-fold intersection
cocycle of $f$ multiplied by $r!/p^{\alpha_{r,p}}$ is trivial.
\end{proposition}


Part (a) is a special case of part (b) for $r=p+1=3$.

The usual form of the \"Ozaydin theorem \cite[Theorem 3.3]{Sk16} states the existence of certain {\it equivariant maps}.
Theorem \ref{vankamz-oz} is equivalent to that statement because the $r$-fold intersection cocycle equals to
the {\it obstruction cocycle} \cite[Lemma 41.b]{MW15} which is a complete obstruction to the existence
of certain equivariant map \cite[Theorem 40]{MW15}.
Analogously Proposition~\ref{vankamz-oz3} is equivalent to the corresponding intermediate result
from the proof of  the `usual' \"Ozaydin theorem.
See simplified exposition in survey \cite[\S3.2]{Sk16}.

It would be interesting to obtain a direct proof of Proposition \ref{vankamz-oz3}, cf. the above direct proofs of Propositions \ref{starz} and  \ref{starz-r}.

\comment

\begin{pr}\label{csgn-r}
(a) The sum of the values of $\nu_f$ over partitions $(R_1,R_2)\in K_S^*$ such that $1\in R_1$ is congruent to the Radon number $\rho(f)$ modulo $2$.

(b REPLACE ???) Define the function $\csgn:K_S^*\to\{+1,-1\}$, $\csgn(R_1,R_2)=+1$ if and only if $1\in R_1$.
Then $\sum\limits_{(R_1,R_2)\in K_S^*} \nu_f(R_1,R_2)\cdot\csgn(R_1,R_2)\underset4\equiv 2$.
\end{pr}


is the map $\nu_f:K_S^*\to\Z$ such that
$$\nu_f(R_1,R_2)\ =\
\begin{cases}
\deg_{f(R_j)}f(R_i), & \text{if}\ |R_i|=1,\ |R_j|=3\ \text{for some }i\ne j;\\
\sum\limits_{X\in fR_1\cap fR_2} \sgn X, & \text{if}\ |R_1|=|R_2|=2.
\end{cases}$$
Here if $R_i=\{A,B\}, A<B$ then $f(R_i)$ is the image of the edge $R_i$ oriented from $f(A)$ to $f(B)$,
and if $R_i=\{A,B,C\},A<B<C$, then $f(R_i)$ is the image of the cycle $R_i$ oriented as $f(A)f(B)f(C)f(A)$.

For a vertex $A$ and an edge $BC,\ B<C,\ B,C\ne A$ define {\bf an elementary skew-symmetric coboundary} of the pair $(A,BC)$ as the map $\delta(A,BC):K_S^*\to\Z$ which assigns

$\bullet$ $1$ to any pair of edges $(R_1,R_2)$ with $R_1=\{B,C\}$ and $R_2$ issuing out of $A$ or with $R_2=\{B,C\}$ and $R_1$ going to $A$,

$\bullet$ $-1$ to any pair of edges $(R_1,R_2)$ with $R_1=\{B,C\}$ and $R_2$ going to $A$ or with $R_2=\{B,C\}$ and $R_1$ issuing out of $A$,

$\bullet$ $1$ to any pair $(R_1,R_2)$ such that $R_i=\{A\}$ and $R_j=\{B,C,D\}$, where either $D<B$, or $C<D$,

$\bullet$ $-1$ to any pair $(R_1,R_2)$ such that $R_i=\{A\}$ and $R_j=\{B,C,D\}$, where $B<D<C$,

$\bullet$ $0$ to any other pair.


(R_1,R_2,R_3)$ for every spherical  partition $(R_1,R_2,R_3)\in \mathrm{S}$


Given definitions, presumably part (a) would be proved by taking $f$ a regular heptagon in the plane with vertices ordered counterclockwise and close to vertices of a triangle,
In the plane take a (non-degenerate) a triangle $ABC$ and seven points $1,\ldots,7$, of which points $1,6$
are close to $A$, points $2,5$ are close to $B$, points $3,4,7$ are close to $C$ (cf. Example \ref{ratv-2r}).

\begin{lemma}[triple intersection invariant]
Take a chessboard coloring of spherical partitions $R$ of $[7]$ such that $7\in T_3$, and denote by $\csgn(R)\in\{+1,-1\}$ the color of $R$ (see Lemma \ref{ratv-ext}.c).
Take a . Then

It is also possible to prove Lemma \ref{p:boundary} differently, using more abstract result, 3-fold Borsuk-Ulam Theorem \ref{t:3bu} below.
Then it turns out that Theorem~\ref{ratv-tvpl} can be proved (see below) without use of Lemma \ref{p:boundary}.
(Der Mohr hat seine Arbeit getan, der Mohr kann gehen.)%
\footnote{Note that use of a degree in 3-fold Borsuk-Ulam Theorem \ref{t:3bu} requires independence of a general position point chosen in the definition of degree, which corresponds to a proof that $V(f)$ modulo 3 is independent of $f$ in the other proof of Lemma \ref{p:boundary} mentioned above.}

, and $fT_s$ is defined analogously.

[cf. alternative proof after Proposition \ref{vkam2-deg2}]

Unlike a direct generalization of Lemmas \ref{11-vankam} and \ref{ratv-vk2} (analogous to  \cite{BMZ15}) we do not bother about the exact values of $\varepsilon_7$, which simplifies the proof.
However, the exact values of $\varepsilon_7$ give

There is a map $\varepsilon_7:\mathrm{S}\to\{+1,-1\}$ such that
$$\sum\limits_{R=(R_1,R_2,R_3)\in \mathrm{S},\atop |R_1|+1=|R_2|=|R_3|-1=2}
\varepsilon(R)\overline{fR_1}_+\cdot fR_2\cdot fR_3+
\sum\limits_{R=(R_1,R_2,R_3)\in \mathrm{S},\atop |R_1|=|R_2|=|R_3|=2}
\varepsilon(R) fR_1\cdot(fR_2\circ fR_3)\underset3\equiv$$


\begin{proof}[Proof of (a)]
Take any $T=(T_1,T_2,T_3)\in \mathrm{S}$ and denote by the same letters $T_1,T_2,T_3$ the corresponding simplices of $\Delta_5$.
Consider the equality for $fT_1\cdot fT_2\cdot fT_3$ given by Proposition \ref{a:3red}.ce.
Sum up all the equalities with signs $\varepsilon_7$ to obtain the right-hand term of the congruence.
By Lemma \ref{ratv-ext}.c and Assertion \ref{p:ext}??? for some choice of the map $\varepsilon_7$ all the summands corresponding to spherical partitions $(R_1,R_2,R_3)$ of $7-\{j\}$ with $j\ne7$ will cancel.\footnote{In fact, the map $\varepsilon_7$ comes from Lemmas \ref{ratv-ext}.c and is independent of $f$.
Instead of Lemma \ref{ratv-ext}.c and Assertion \ref{p:ext}??? we can use Lemmas \ref{cont-hom},\ref{cont-homeq}.}
The remaining sum is the left-hand term of the congruence.
\end{proof}






For every $j\in[6]$ and spherical partition $R$ of $[6]$ let $[R:j]=-1$ if $j$ is the second
value of some set from $R$ (w.r.t. increasing of values), and let $[R:j]=+1$ otherwise.
There is a unique map
$$\sgn:\mathrm{S_6}\to\{+1,-1\}$$
from the set $\mathrm{S_6}$ of all spherical partitions of $[6]$ such that

$\bullet$ $\sgn(\{1,2,3\},\{4,5,6\},\emptyset)=+1$.

$\bullet$ for every $j\in[6]$ and every spherical partition $G$ of $[6]-\{j\}$ take
the two spherical partitions $R,R'$ of $[6]$ extending $G$ given by (a).
Then $[R:j]\sgn R+[R':j]\sgn R'=0$.
\footnote{Take orientation $ab$ on every 2-element subset $\{a,b\}\subset[7]$, $a<b$.
Take orientation $abc$ on every 2-element subset $\{a,b,c\}\subset[7]$, $a<b<c$.
Consider the {\it join orientation on 5-cells of $\Delta_5*\Delta_5*\Delta_5$} corresponding to spherical partitions.
Then these cells $R$ with coefficients $\sgn R$ form an {\it integer 5-cycle}.
It would be interesting to know if the map $\sgn$ coincides with the combinatorial sign $\csgn$ of  \cite[\S5]{MTW12}, \cite{BMZ15}, or is the second bullet point condition equivalent to $\csgn R+\csgn R'=0$.}

\begin{lemma}\label{ratv-csgn}
(a) For every general position PL map $f:K_7\to\R^2$ and spherical partition $(R_1,R_2,R_3)$ we have
$\sgn(R_1,R_2,R_3) (fR_1\cdot fR_2\cdot fR_3) = \sgn(R_2,R_1,R_3) (fR_2\cdot fR_1\cdot fR_3)$.

(b) The triple van Kampen number does not change under the Reidemeister moves in fig.~\ref{reidall}.I-V.
\end{lemma}

In other words, assign to a spherical partition $(R_1, R_2, R_3)$ the spherical partition $(R_2, R_1, R_3)$.
This correspondence is a bijection.
The corresponding summands of the Triple Intersection Invariant Lemma are equal.
Then the number of nonzero summands in the formula is even, so $v(f) = 0$.

(c) For each distinct $a,b,c,d,e\in[7]$ such that $7\not\in\{a+b,c+d,d+e,e+c\}$ we have
$$\varepsilon_{ab,cde}=\varepsilon_{b,cde}=-\varepsilon_{a,cde}\quad\text{where}\quad \varepsilon_{(S,T)}=\sgn(R_1,R_2,R_3)$$
for $\{R_1,R_2,R_3\}=\{S,T,[7]-S-T\}$ and $\max R_1<\max R_2<\max R_3$.

\begin{proof}[???Sketch of deduction of Problem~\ref{p:boundary} from known results]
Orient every 2-element subset $\{a,b\}\subset[7]$, $a<b$, as $(ab)$.
Orient every 3-element subset $\{a,b,c\}\subset[7]$, $a<b<c$, as $(abc)$.
Take any general position PL map $K_7\to\R^2$.
The value of the reduction $V_{r,f}$ modulo 3 of the `join' threefold intersection cocycle on a spherical partition $(R_1,R_2,R_3)$ of $[7]$ is $V_{r,f}(R_1,R_2,R_3)=(-1)^{|R_1|\cdot|R_2|\cdot|R_3|}(fR_1\cdot fR_2\cdot fR_3)$.
Denote by $X\subset\Delta_6*\Delta_6*\Delta_6*$ the union of simplices corresponding to spherical partitions $(R_1,R_2,R_3)$ of $[7]$ (not necessarily $7\in R_3$).
In proofs of \cite{VZ93}, \cite[\S6.5]{Ma03} it is shown that the restriction $V_{3,f}:X\to\Z_3$ is non-trivial.
A slight improvement of Proposition \ref{ratv-ext} shows that $H_6(X/Z_3;\Z_3)\cong\Z_3$ is generated by the of the 6-chain $\sum\limits_{R=(R_1,R_2,R_3)\in X}\csgn R(R_1*R_2*R_3)$.
Hence the scalar product of this 6-chain with $V_{3,f}$ is non-zero.
This scalar product is the number of the lemma.
The signs have to be checked.
\end{proof}



\footnote{Denote $R:=(\emptyset,[3],[7]-[3])$.
For a permutation $\sigma\in\Sigma_3$
denote by $R_\sigma$ the partition obtained from $R$ by permuting the three subsets $\emptyset,[3],[7]-[3]$ according to $\sigma$, and, if 7 is not in the third set after permutation, moving 7 to the (new) third set.
In \cite{MTW12} (`Sarkaria-Onn') maps $S_\sigma:K^{\underline{3}}\to\Z_3$
were constructed so that $\sum_{\sigma\in\Sigma_3} S_\sigma\cdot V_{r,f}\ne0\in\Z_3$ for every $f$. NO!}



\begin{example}\label{n}
(a) Suppose a general position PL map $f:K_7\to\R^2$ is given.
Denote by $\t\Delta_6^{(4)}$ the set of ordered partitions $[7]=R_1\sqcup R_2\sqcup R_3$ (into oriented sets)
such that either $|R_1|=|R_2|=2$, or $|R_i|=1,|R_2|=3$ for some $i\ne j$.
Then the threefold intersection cocycle is a map $\t\Delta_6^{(4)}\to\Z$, whose value on the partition $(R_1,R_2,R_3)$ is
$$\begin{cases}
(fR_j\cdot fR_i)(fR_k\cdot fR_i) &  |R_i|=1,\ |R_j|=|R_k|=3\text{ for some }\{i,j,k\}=\{1,2,3\} \\
-\sum\limits_{X\in f(R_i)\cap f(R_j)} \sgn X (fR_k\cdot X) & |R_i|=|R_j|=2,\ |R_k|=3 \text{ for some }\{i,j,k\}=\{1,2,3\}
\end{cases}.$$
\end{example}

Part (b) is essentially proved in \cite{Oz}, .
The plan of the proof is following.
Every PL map $K^{(1)}\to\R^2$ that have no Radon points corresponds to some {\it equivariant map}
$\t K^r\to S^{2r-3}$ with respect to action of the symmetric group $\mathfrak{S}_r$.
Such an equivariant map exists if and only if {\it the $r$-obstruction cocycle of $K$},
that is a certain element of $H^{2r-2}(K)$, is zero.
If $r$ is not a power of a prime $p$, then there is a $G_p$-equvariant map $\t K^r\to S^{2r-3}$,
where $G_p\subset \mathfrak{S}_r$ is a Sylow $p$-subgroup that has index $r!/p^{\alpha_{r,p}}$.
Equivariant obstruction theory allows to deduce that the $r$-obstruction cocycle
multiplied by $r!/p^{\alpha_{r,p}}$ is null-cohomologous.
Finally, it turns out that the $r$-obstruction cocycle equals to the $r$-fold intersection cocycle,
see \cite[\S3.3]{Sk16}.

For a finite subset $X\subset[7]$ denote by $X^-$ the sequence obtained from $X$ by deleting the minimal element and ordered by increasing.
For a spherical partition $R=(R_1,R_2,R_3)$ of $[7]$ denote by $R^-$ the sequence obtained by writing the sequences $R_1'^-,R_2'^-,R_3'^-$ in this order.
E.g. $(\{1,2,3\},\{4,5,6\},7)^-=(2,3,5,6)$.

\begin{pr}\label{ratv-csgn}
(a) For every $j\in[6]$ and every spherical partition $S$ of $[7]-\{j\}$ take
the two spherical partitions $R,T$ of $[7]$ extending $S$ (see assertion \ref{ratv-ext}).
Then there are $k,l\in[4]$ such that the sequence obtained from $R^-$ by deleting $k$-th element is the same as the sequence obtained from $T^-$ by deleting $l$-th element.

(b)* 
There is a unique map $\csgn:\mathrm{S}\to \{+1,-1\}$ from the set $\mathrm{S}$ of all spherical partitions of $[7]$ such that

$\bullet$ $\csgn(\{1,2,3\},\{4,5,6\},\{7\})=+1$.

$\bullet$ For every $j\in[6]$ and every spherical partition $S$ of $[7]-\{j\}$ take
the two spherical partitions $R,T$ of $[7]$ extending $S$ (see assertion \ref{ratv-ext}).
Then $\csgn R\cdot\csgn T=(-1)^{k-l}$, where $k,l$ are defined in (a).

(c)* $\csgn(R_2,R_1,R_3)=\begin{cases} -\csgn(R_1,R_2,R_3) & |R_1|=|R_2|=2 \\
\csgn(R_1,R_2,R_3) & \text{otherwise}\end{cases}$.

\end{pr}

For every partition $R=(R_1,R_2,R_3)$ of the ordered set $(1,\ldots,7)$ of general position points in the (oriented) plane define the {\bf triple intersection sign} $\gsgn R$ as follows:

$\bullet$ $\gsgn R=+1$ if, up to a permutation of $(R_1,R_2,R_3)$ which does not change the order of two-element sets,%
\footnote{Observe that a permutation of a spherical partition need not be a spherical partition.}

\quad $\circ$ either $R_1=\{a,b,c\}$, $R_2=\{d,e,f\}$, $a<b<c$, $d<e<f$, and plane triangles $abc$, $def$ have the same orientation;

\quad $\circ$ or $R_1=\{a,b,c\}$, $R_2=\{d,e\}$, $R_3=\{f,g\}$, $a<b<c$, $d<e$, $f<g$ and plane triangles $abc$, $deg$ have the opposite orientations;

$\bullet$ $\gsgn R=-1$ otherwise.

\begin{pr}\label{ratv-csgn1}
$\gsgn(R_2,R_1,R_3)=\begin{cases} -\gsgn(R_1,R_2,R_3) & |R_1|=|R_2|=2 \\
\gsgn(R_1,R_2,R_3) & \text{otherwise}\end{cases}$.
\end{pr}

\endcomment

\subsection{Appendix: some details to \S\ref{s:mucoge}}\label{s:apptve}

{\bf \ref{0-radpl}.}
If some three of these points lie on a line, then one of the points lies in the segment with vertices at the two other points.
If no three of these points lie on a line, then use Proposition \ref{ratv-vk2l}.


\smallskip
{\bf \ref{ratv-vk2l}.}
The convex hull of 4 points in general position is either a triangle or a quadrilateral.
In the case of a triangle, 
only partition $\{\text{the vertices of the triangle},\ \text{the remaining point}\}$ has a nonempty intersection of the convex hulls of parts.
In the case of a quadrilateral, 
only partition $\{\text{the ends of the first diagonal},\ \text{the ends of the second diagonal}\}$ has a nonempty intersection of the convex hulls of parts.

\smallskip
{\bf \ref{ratv-9ex}.} (d) Let us show that these points satisfy the required condition.
Observe that for each pair of intersecting segments there exists an isometry mapping from this pair to one of the following pairs: $\{AB_1, BA_1\}, \{AO, A_1B_1\}, \{AO, BA_1\}$.
Also $A_1 = AO \cap A_1B_1 = AO \cap A_1B$ belongs neither to $\bigtriangleup BC_1C$ nor to $\bigtriangleup B_1C_1C$, and $AB_1 \cap BA_1$ does not belong to $\bigtriangleup$ $OC_1C$.

\begin{figure}[h]\centering
\includegraphics[scale=0.9]{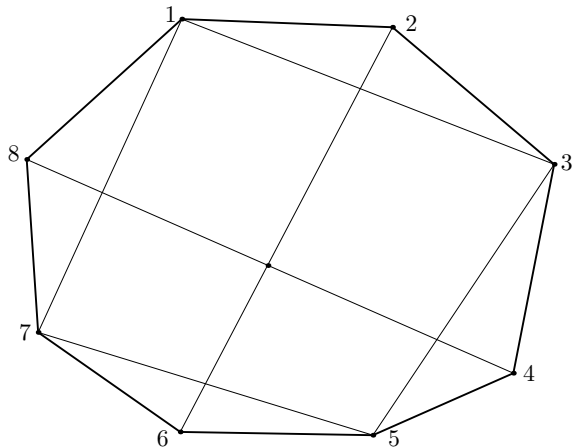}
\caption{The convex octagon}
\label{octagon}
\end{figure}

\smallskip
{\bf \ref{ratv-9}.}
(a) Denote the vertices of the octagon by $1\ldots 8$ (in clockwise order, see fig.~\ref{octagon}).
Decompose the vertices into three sets $\{1,3,5,7\}$, $\{2,6\}$ and $\{4,8\}$.
The convex hulls of these sets are quadrilateral $1357$ and segments $26$, $48$, respectively.
Clearly, segments $26$ and $48$ have an intersection point, say $A$.
Clearly, $A$ does not belong to any of the triangles $123$, $345$, $567$, $781$.
Hence the three convex hulls have a common point.


(b) If the convex hull $Z$ of given 11 points has at least eight vertices, then we can decompose these points into three sets using (a).
Otherwise, if $Z$ has less than eight vertices, then let the first set $S_1$ of our partition be the set of these vertices.
Since at least 4 points are left, we can decompose them into two sets with intersecting convex hulls.
The intersection point also belongs to the convex hull of $S_1$.

(c) {\it (Written by I. Bogdanov.)}
Let $N=9r+1$.
Consider convex hulls for each $6r+1$ of required points.
Any three of these convex hulls intersect because they have a common point.
Therefore by the Helly theorem they all have a common point (not necessary one of the preceding points).
Denote by $O$ this common point.
Let us prove that $O$ a common point of convex hulls of some $r$ disjoint subsets.

(The following elegant property of $O$ we will not use: on both sides of any line passing through $O$ there are at least $3r+1$ of the given points.)

Point $O$ belongs to the convex hull of any $6r+1$ of these given points.
In other words, from any $6r+1$ of the given points we can choose three points such that the triangle with vertices at this points contains point $O$.
Let us choose this triangle, remove it, choose new triangle, etc.
This can be done at least $r$ times.


\smallskip
\textbf{\ref{ratv-chess}.} (a) Take any point $A\not\in L$ in general position with the vertices of $L$ and paint it black.
For any point $B\not\in L$ in general position with $A$ and the vertices of $L$
take a polygonal line $S$ in general position with $L$ joining $B$ to $A$.
Paint $B$ black if the number $|S\cap L|$ is even and white otherwise.

If $B$ is not in general position with $A$ and the vertices of $L$,
take a small (2-dimensional) triangle disjoint from $L$ and whose interior contains $B$.
The triangle contains some points which was colored in previous paragraph.
Obviously, their colors coincide.
Paint $B$ in the same color.

By the Parity Lemma \ref{0-even}.b this coloring is well-defined, i.e. does not depend on the choice of $S$ and of the disk.

When the point $B$ passes to the adjacent domain, the number $|S\cap L|$ changes parity.
So the adjacent domains have different colors.


(b) If we replace $A$ by a close point, the coloring does not change.
So we may assume that $A$ and the vertices of $L,P$  are in general position.
Join $A$ to the ends of $P$ by polygonal lines $S_1$ and $S_2$ such that the vertices of $L,P,S_1,S_2$ are in general position.
By the Parity Lemma \ref{0-even}.b the number $|L\cap(S_1\cup P\cup S_2)|$ is even.
Hence $|S_1\cap L|+|S_2\cap L|\underset2\equiv |P\cap L|$.
So we have the required statement.



\begin{proposition}\label{vkam2-deg2}
In the plane take non-closed polygonal lines $P$ and $Q$ whose vertices are in general position.


(a) For any pair $(p,q)\in\partial(P\times Q):=(\partial P\times Q)\cup(P\times \partial Q)$ we have $p\ne q$
because vertices of $P$ and $Q$ are in general position.

(b) Take corresponding PL maps $p,q:[0,1]\to\R^2$.
Then the number $|P\cap Q|$ has the same parity as the number of rotations of the vector $p(x)-q(y)$ while $(x,y)$ goes along the boundary $\partial([0,1]^2)$ of the square $[0,1]^2$, and the number $P\cdot Q$ equals to that  number of rotations.
\end{proposition}

\begin{proof}[Sketch of a proof] Part (a) is obvious.
If $p(t)\notin Q$, then one can define (not as an integer) the number of rotations for the restriction $p|_{[0,t]}$.
If $p(t)\in Q$, then these numbers of rotations for restrictions $p|_{[0,t-\varepsilon]}$ and $p|_{[0,t+\varepsilon]}$ differ by $\pm1$ (depending on the sign of the intersection point, see  fig.~\ref{f:sign}).
This argument for (b) becomes rigorous after formally defining `the number of rotations'.
\end{proof}



\begin{remark}\label{r:rasp}
An ordered partition $(R_1,\ldots,R_r)$ of $[3r-2]=R_1\sqcup\ldots\sqcup R_r$ into $r$ sets (possibly empty) is called {\it rainbow} if for every $j=1,\ldots,r$ the set $R_j$ intersects each of the three sets $\{1,\ldots,r-1\},\{r,\ldots,2r-2\},\{2r-1,\ldots,3r-3\}$ by at most one element.
Theorem \ref{ratv-tv7} is true for `spherical' replaced by `rainbow' \cite[Theorem 2.2]{BMZ15}, \cite[Theorem 2]{MTW12}.
Observe that `spherical' is the same as `rainbow'  for $r=3$ but is different for $r>3$.
So the proofs of the topological Tverberg Theorem \ref{ratv-tvpl} from \S\ref{s:ttp} (based on \cite{VZ93}) and from \cite{BMZ15, MTW12} are different not only in exposition, they give different improvements.
For a generalization potentially containing yet another alternative proof of Theorems \ref{ratv-tvpl}, \ref{ratv-tv7} see~\cite{BMZ11}.
\end{remark}


Proofs of the following assertions are not hard and are omitted.

\begin{pr}\label{a:3red}
In the plane take an oriented line $l$, non-closed polygonal lines $P,Q$, point $R$ and the modulo 2 interiors  $S,T$ of closed polygonal lines $\partial S,\partial S$, so that no line joining two vertices of the polygonal lines or $R$ is parallel to $l$.
Denote by $\overline X$ the line parallel to $l$ passing through a point $X$, and by $\overline X_\pm$
the two rays of $\overline X$ starting at $X$.

(a) We have
$$|R\cap S| \underset2\equiv |\overline R_+\cap\partial S| \underset2\equiv |\overline R_-\cap\partial S|
\quad\text{and}\quad |P\cap Q| \underset2\equiv |\overline{\partial P}_+\cap Q|+|P\cap \overline{\partial Q}_-|.$$

(b) We have
$$(R\cdot\partial S)(R\cdot\partial T)\underset2\equiv
|\overline R_+\cap S\cap\partial T|+|\overline R_+\cap \partial S\cap T|.$$

(c) If both $\partial S$ and $\partial T$ are oriented, then
$$(R\cdot\partial S)(R\cdot\partial T)=\overline R_+\cdot\partial S\cdot T+\overline R_+\cdot\partial T\cdot S,$$
where $\partial S$ and $\partial T$ are considered as non-closed polygonal lines.


(d) We have
$$P\cdot Q\cdot S\underset2\equiv|Q\cap(P\circ\partial S)|+|\partial Q\cap(P\circ S)|+|Q\cap(\partial P\circ S)|,
$$
where for subsets $X,Y\subset\R^2$ we denote
$$X\circ Y:=\{(x+y)/2\ :\ x\in X,\ y\in Y,\ y-x\uparrow\uparrow l\}.$$

(e) If $P,Q$ and $\partial S$ are oriented, then $P\circ S$ is a closed naturally oriented polygonal line and
$\partial P\circ S$ is a union of two such polygonal lines.
We have
$$\pm P\cdot Q\cdot S = \partial Q\cdot(P\circ S) + Q\cdot(P\circ\partial S) + Q\cdot(\partial P\circ S),
$$
where we write $\partial S$ to emphasize that $P\circ\partial S$ is considered as a non-closed polygonal line.
\end{pr}

\begin{figure}[h]
\centerline{\includegraphics[width=6cm]{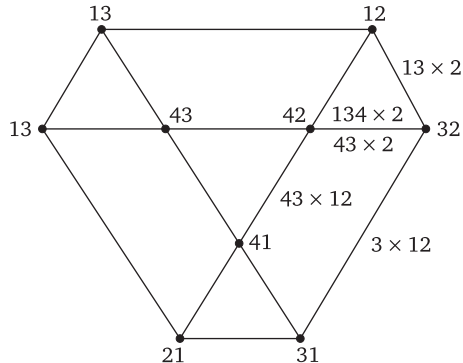}}
\caption{Configuration space}\label{f:del3}
\end{figure}

\begin{proof}[An alternative proof of (a version of) Lemma \ref{ratv-vk2}]
(There is analogous proof of Lemma \ref{11-vankam}.)
Take an oriented line $l$ in the plane and a general position PL map $f:K_4\to\R^2$ such that no line joining two vertices of the polygonal lines forming $f$ is parallel to $l$.
Denote by $ijk\times q$ and $ij\times kq$ the equalities of Assertion \ref{a:3red}.a corresponding to
$f(ijk)\cap f(q)$ and to $f(ij)\cap f(kq)$, respectively.
Sum up the 7 equalities
$$124\times3,\quad 4\times123,\quad 134\times2,\quad 234\times1,\quad
43\times12,\quad 24\times13,\quad 14\times23$$
(cf. figure \ref{f:del3}).
We obtain
$$\rho(f)\underset2\equiv|\overline{f(1)}\cap f(23)|+|\overline{f(2)}\cap f(13)|+|\overline{f(3)}\cap f(12)|.$$
The parity of the right-hand number clearly depends
only on the order of the projections of $f(1),f(2),f(3)$ along $l$ to some line.
Then clearly $\rho(f)=1\in\Z_2$.
(Cf. Remark \ref{vkam-line}.)
\end{proof}

The definition of a degree can be found e.g. in \cite[\S2.4]{Ma03}, \cite[\S8]{Sk20}.

\begin{pr}\label{a:3redmap} Use the notation of Assertion \ref{a:3red}.
Assume that $(n_1,n_2,n_3)$ is either $(0,2,2)$ or $(1,1,2)$.

If $n_1=0$, then define $f_1:D^0\to\R^2$ by setting $f(0):=R$.

If $n_1=1$, then take a general position PL map $f_1:D^1\to\R^2$ whose image is $P$.

If $n_2=1$, then take a general position PL map $f_2:D^1\to\R^2$ whose image is $Q$.

If $n_2=2$, then take a general position PL map $f_2:D^2\to\R^2$ such that $f_2(\partial D^2)=\partial S$ (but possibly $f_2(D^2)\ne S$).

If $n_3=2$, then take a general position PL map $f_3:D^2\to\R^2$ such that $f_3(\partial D^2)=\partial T$.

(a) Then for each $(t_1,t_2,t_3)\in D:=D^{n_1}\times D^{n_2}\times D^{n_3}\cong D^4$ the points $f_1t_1,f_2t_2,f_3t_3\in\R^2$ are not all equal (although two of them could coincide).

(b) For $x_1,x_2,x_3\in\R^2$ which are not all equal define
$$
S:=x_1+x_2+x_3,\quad \pi':=\left(x_1-\frac S3,x_2-\frac S3,x_3-\frac S3\right)
\quad\text{and}\quad \pi:=\frac{\pi'}{|\pi'|}.
$$
This defines a map
$$
\pi:(\R^2)^3-\diag\to S^3_{\Sigma_3},\quad\text{where}\quad
\diag:=\{(x,x,x)\in(\R^2)^3\ :\ x\in\R^2\}.
$$
So we obtain the map
$$
\pi\circ(f_1\times f_2\times f_3):\partial D\to S^3_{\Sigma_3}.
$$
The degree of this map (for the boundary of the product orientation of $\partial D\cong S^3$) equals either to $R\cdot S\cdot T$ or to $P\cdot Q\cdot S$ when $(n_1,n_2,n_3)$ is either $(0,2,2)$ or $(1,1,2)$, respectively.

(c) Let $D^*:=D^{n_1}*D^{n_2}*D^{n_3}\cong D^6$.
The degree of the map $\pi^*\circ(f_1*f_2*f_3):\partial D^*\to S^5_{\Sigma_3}$ defined after Theorem \ref{t:3bu}
(for the boundary of the join orientation of $\partial D^*\cong S^5$) equals either to $R\cdot S\cdot T$ or to $P\cdot Q\cdot S$ when $(n_1,n_2,n_3)$ is either $(0,2,2)$ or $(1,1,2)$, respectively.
\end{pr}

\begin{pr}\label{ratv-ext} (a) For every $j\in[6]$ and spherical partition $G$ of $[6]-\{j\}$ there are exactly two spherical partitions of $[6]$ extending $G$.

(b) The set of all spherical partitions of $[6]$ admits a chessboard coloring, i.e. a coloring in two colors such that for every $j\in[6]$ and spherical partition $G$ of $[6]-\{j\}$ the two spherical partitions of $[6]$ extending $G$ have different colors.



(c) A 5-simplex of $\Delta_5*\Delta_5*\Delta_5$ corresponding to a partition $R$ of $[6]$
contains a 4-simplex of $\Delta_5*\Delta_5*\Delta_5$ corresponding to a partition $G$ of $[6]-\{j\}$
if and only if $R$ extends $G$.

(d) Take the 5-simplex of $\Delta_5*\Delta_5*\Delta_5$ corresponding to a spherical partition $(R_1,R_2,R_3)$ of $[6]$.
If $j\in R_i$, then denote by $j_i$ the corresponding vertex of the 5-simplex.
Orient the 5-simplex as $(1_{i_1},2_{i_2},\ldots,6_{i_6})$, where $j\in R_{i_j}$.
Then such orientations of two 5-simplices having a common 4-simplex disagree along this 4-simplex.


(e) Assume that a triangulation of an $n$-manifold and a collection of orientations on $n$-faces is given, so that these orientations disagree along every $(n-1)$-face.
Assume further that the faces admit a chessboard coloring.
Then the manifold is orientable.
\end{pr}

\begin{proof}[Proof of (a)] The number $j$ can be added to two among three sets in the partition $G$ because exactly one of the sets contains the `twin' $7-j$ of $j$ which cannot appear together with $j$.%
\footnote{The proof of \cite[Lemma 8]{MTW12} essentially proves parts (a,b) (for (a) the chessboard terminology is not required, see above).
The first two sentences of \cite[Lemma 8]{MTW12} is a statement similar to parts (a,b)  in more sophisticated terminology involving Sarkaria-Onn transform (and in more generality).
The parenthetical remark of \cite[Lemma 8]{MTW12} `(In topological terminology, this is the orientable pseudomanifold property)' does not follow from the first two sentences.
However, this remark is really a remark not part of the formal statement of that lemma, and is not formally used later in \cite{MTW12}.
Also, this remark does follow from the first two sentences of \cite[Lemma 8]{MTW12} together with
(generalizations of) Sarkaria-Onn versions of (d,e).
Note that the orientable pseudomanifold of \cite[Lemma 8]{MTW12} in the case corresponding to the Topological Tverberg Theorem \ref{ratv-tvpl} (i.e. to $d=2$, $C_1=\{1,\ldots,r-1\}$, $C_2=\{r,\ldots,2r-2\}$,
$C_3=\{2r-1,\ldots,3r-3\}$, $C_4=\{3r-2\}$) is  a join of three orientable pseudomanifolds, each of them being the {\it chessboard complex} $\Delta_{r,r-1}$ with parameters $r-1,r$.}
\end{proof}


\section{Conclusion: higher-dimensional generalizations}\label{s:high}

\subsection{Radon, Tverberg and van Kampen--Flores theorems}\label{s:hirtvkf}

A subset of $\R^d$ is {\it convex},
if for any two points from this subset the segment joining these two points is in this subset.
The {\it convex hull} of a subset $X\subset\R^d$ is the minimal convex set that contains $X$.

\begin{theorem}[Radon, cf. Theorem \ref{0-radpl}]\label{t:lr}
For every integer $d>0$ any $d+2$ points in $\R^d$ can be decomposed into two groups whose convex hulls intersect.
\end{theorem}


\begin{theorem}[Linear van Kampen--Flores, cf. Proposition \ref{0-ra2}.a]\label{t:lvkf}
For every integer $k>0$ from any $2k+3$ points in $\R^{2k}$ one can choose two disjoint $(k+1)$-tuples whose convex hulls intersect.
\end{theorem}

This implies simplicial non-embeddability in $\R^{2k}$ of the complete
$k$-complex on $2k+3$ vertices.

\begin{theorem}[Tverberg; cf. Theorem \ref{ratv-tv1}]\label{t:lt}
For every integers $d,r>0$ any $(d+1)(r-1)+1$ points in $\R^d$ can be decomposed into $r$ groups
whose $r$ convex hulls have a common point.
\end{theorem}


For a motivated exposition of the well-known proof see \cite{RRS}.

Here the number $(d+1)(r-1)+1$ could be remembered by remembering the following simple examples showing that this number is the least possible\jonly{ \cite[Excercise 2 to \S6.4]{Ma03}}.
Clearly, every $(d+1)(r-1)$ points in general position in $\R^d$ (or vertices of a $d$-dimensional simplex taken with multiplicity $r-1$, cf. Example \ref{ratv-9ex}.b) do not satisfy the property of the Tverberg Theorem \ref{t:lt}.
Analogous remark holds for Theorems \ref{t:lr}, \ref{t:lvkf}, \ref{t:tr}, \ref{t:tvkf} and
(the proved case when $r$ is a power of a prime) of Conjecture \ref{c:tt}.

\begin{conjecture}[Linear $r$-fold van Kampen--Flores]\label{c:lvkfg}
For every integers $k,r>0$ from any $(r-1)(kr+2)+1$ points in $\R^{kr}$ one can choose $r$ pairwise disjoint $(k(r-1)+1)$-tuples whose $r$ convex hulls have a common point.
\end{conjecture}

This is true for a prime power $r$ \cite{Vo96v} and is an open problem for other $r$ \cite[beginning of \S2]{Fr17}.

Here the number $(r-1)(kr+2)+1$ could be remembered by remembering the following simple examples showing that
this number is the least possible.
Take in $\R^{kr}$ the vertices of a $kr$-dimensional simplex and its center.
Either take every of these $kr+2$ points with multiplicity $r-1$ or for every point take close $r-1$ points in general position.
We obtain $(r-1)(kr+2)$ points in $\R^{kr}$ such that for any $r$ pairwise disjoint $(k(r-1)+1)$-tuples all the $r$ convex hulls of the tuples do not have a common point.
(For $r=3k=3$ cf. \cite[Example 6.7.4]{Ma03}: `It is not known whether such triangles can always be found for 9 points in $\R^3$'.)


Denote by $\Delta_N$ the $N$-dimensional simplex.

\begin{theorem}[Topological Radon theorem, \cite{BB}, cf. Theorem \ref{ratv-totv}]\label{t:tr}
For any continuous map $\Delta_{d+1}\to\R^d$ there are two disjoint faces whose images intersect.
\end{theorem}

\begin{theorem}[van Kampen--Flores, cf. Theorem \ref{grapl-nonalm}]\label{t:tvkf}
For any continuous map
$\Delta_{2k+2}\to\R^{2k}$ there are two disjoint $k$-dimensional faces whose images intersect.
\end{theorem}

This implies PL non-embeddability in $\R^{2k}$ of the complete
$k$-complex on $2k+3$ vertices.

The Topological Radon and the van Kampen--Flores Theorems \ref{t:tr} and \ref{t:tvkf} generalize Radon and the Linear van Kampen--Flores Theorems \ref{t:lr} and \ref{t:lvkf}.
These results are nice in themselves, and are also interesting because they are corollaries of the celebrated Borsuk-Ulam Theorem (see e.g. \cite[\S2.1]{Sk16}), of which the topological Radon Theorem \ref{t:tr} is also a simplicial version.
The PL (piecewise-linear) versions of the Topological Radon and the Linear van Kampen--Flores Theorems \ref{t:tr} and \ref{t:tvkf} are as interesting and non-trivial as the stated topological versions, see Remark \ref{r:ae}.

The above results have `quantitative version' analogous to Propositions \ref{0-ra2}.b and \ref{ratv-vk2l}, Lemmas \ref{11-vankam} and \ref{ratv-vk2}, see e.g. \cite[\S4]{Sk16}.
For direct proofs of some implications between these results see \cite[\S4]{Sk16}.

\begin{conjecture}[topological Tverberg conjecture]\label{c:tt}
For every integers $r,d$ and any continuous map $f\colon\Delta_{(d+1)(r-1)}\to \R^d$ there are pairwise disjoint faces $\sigma_1,\ldots,\sigma_r\subset\Delta_{(d+1)(r-1)}$ such that
$f\sigma_1\cap \ldots \cap f\sigma_r\ne\emptyset$.
\end{conjecture}

This conjecture generalizes both the Tverberg and the topological Radon Theorems \ref{t:lt} and \ref{t:tr}.
This conjecture is true for a prime power $r$ \cite{BSS, Oz, vo96}, is false for $r$ not a prime power and $d\ge2r+1$ by Remark \ref{r:ae}.a and Theorem \ref{t:ozmawaco}.a below, and is an open problem for $r$ not a prime power and $d\le2r$ (e.g. for $d=2$ and $r=6$).




\begin{conjecture}[$r$-fold van Kampen--Flores]\label{c:vkfg}
For every integers $r,k>0$ and any continuous map $f\colon\Delta_{(kr+2)(r-1)}\to \R^{kr}$ there are pairwise disjoint $k(r-1)$-dimensional faces $\sigma_1,\ldots,\sigma_r\subset\Delta_{(kr+2)(r-1)}$ such that
$f\sigma_1\cap \ldots \cap f\sigma_r\ne\emptyset$.
\end{conjecture}

This is true for a prime power $r$ \cite{Sa91g}, \cite[Corollary in \S1]{Vo96v}, is false for $r$ not a prime power and $k\ge2$ by Theorem \ref{t:ozmawaco}.a below, and is an open problem for $r$ not a prime power and $k=1$.

The arguments for results of this subsection form a beautiful and fruitful interplay between combinatorics,
algebra, geometry and topology.
Recall that more motivation, detailed description of references and proofs can found in the surveys mentioned in the `historical notes' of the Introduction.


\subsection{Recognizing realizablity of complexes}\label{s:hialg}

Definition of a $k$-complex (and its relation to hypergraphs) is recalled in \S\ref{s:radc2}.

\UseRawInputEncoding

Realizability of hypergraphs or complexes in the $d$-dimensional Euclidean space $\R^d$ is defined similarly to the realizability of graphs in the plane.
E.g. for 2-complex one `draws' a triangle for every three-element subset.
There are different formalizations of the idea of realizability.


A complex $(V,F)$ is {\bf simplicially} (or linearly) {\bf embeddable} in $\R^d$ if there is a set $V'$ of distinct points in $\R^d$ corresponding to $V$ such that for any subsets $\sigma,\tau\subset V'$ corresponding to elements of $F$ the convex hull $\left<\sigma\right>$ is a
simplex of dimension $|\sigma|-1$ and $\left<\sigma\right>\cap\left<\tau\right>=\left<\sigma\cap\tau\right>$.
\algor{\footnote{This property means that there is an embedded set of simplices in $\R^d$ whose vertices correspond to $V$ and whose simplices correspond to $F$ (an embedded set of simplices of different dimensions in $\R^d$ is defined analogously to \S\ref{0-reaemb}).}
Это свойство формализует <<отсутствие самопересечений>>.}

\begin{theorem}[General Position]\label{0-gp5} Any $k$-complex is simplicially embeddable in $\R^{2k+1}$.
\end{theorem}

Already in the early history of topology mathematicians showed that in the General Position Theorem \ref{0-gp5}
the number $2k+1$ is the least possible.
See the van Kampen--Flores Theorems \ref{t:lvkf} and \ref{t:tvkf}.

\begin{proposition}[cf. Proposition \ref{1-alg}]\label{t:recs} For every fixed $d,k$ there is an algorithm for recognizing the simplicial embeddability of $k$-complexes in $\R^d$.\algor{\footnote{This problem is PSPACE, meaning that there is an algorithm that uses polynomial space for computation (but not polynomial time).
It turns out from the complexity theory that the time is bounded by an exponential function.
Also, whatever problem is solvable in NP, it is also solvable in PSPACE.
But it is conjectured that PSPACE is in general worse than NP and in particular than P.}}
\end{proposition}


\invadraw{Below  $NP$-hardness of an algorithmic problem depending on an integer parameter $n$ means that using a devise which solves this problem at 1 step, we can construct an algorithm which is polynomial in $n$ and which recognizes if a boolean function of $n$ variables is identical zero, the function given as a disjunction of some conjunctions of  variables or their negations (e.g. $f(x_1,x_2,x_3,x_4)=x_1x_2\overline x_3\vee \overline x_2x_3x_4\vee \overline x_1x_2x_4$).}

\algor{Общепринятое формальное определение NP-трудности непросто.
Приведем следующее эквивалентное определение.
Алгоритмическая проблема, зависящая от целочисленного параметра $n$, называется {\it NP-трудной}, если, имея автомат для ее решения за 1 шаг, можно построить полиномиальный по $n$ алгоритм распознавания тождественности нулю булевой формулы от $n$ переменных, являющейся дизъюнкцией конъюнкций переменных и их отрицаний (например,
$f(x_1,x_2,x_3,x_4)=x_1x_2\overline x_3\vee \overline x_2x_3x_4\vee \overline x_1x_2x_4$).}

\begin{conjecture}\label{t:nphhc} For every fixed $d,k$ such that $3\le d\le\frac{3k}2+1$ the algorithmic problem of recognizing simplicial embeddability of $k$-complexes in $\R^d$ is NP hard.\footnote{M. Tancer suggests that it is plausible to approach the conjecture the same way as in \cite{MTW, ST17}.
Namely, one can possibly triangulate the gadgets in advance and glue them together so that the `embeddable gadgets' would be simplicially embeddable for the prescribed triangulations.
By using the same triangulation on gadgets of same type, one can achieve polynomial size triangulation.
Realization of this idea should be non-trivial.}
\end{conjecture}


\begin{figure}[h]\centering
\includegraphics[scale=0.9]{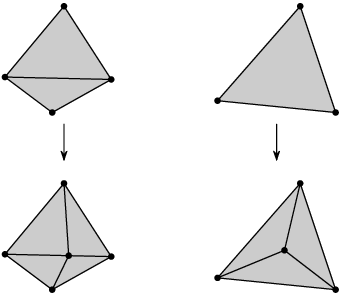}
\caption{Subdivision of an edge}
\label{podra1}
\end{figure}

The {\it subdivision of an edge} operation is shown in fig.~\ref{podra1} left (exercise: represent
the {\it subdivision of a face} operation is shown in fig.~\ref{podra1} right as composition of several subdivisions of an edge and inverse operations).
A {\bf subdivision} of a complex $K$ is any complex which can be obtained from $K$ by several subdivisions of an edge.

\UseRawInputEncoding

A complex is {\bf PL} (piecewise linearly) {\bf embeddable} in $\R^d$ if some its subdivision is simplicially embeddable in $\R^d$.\footnote{The related different notion of being topologically
embeddable \algor{(\S\ref{s:topnr})} is not required in this text\algor{ outside Remark \ref{r:pltop} and some parenthetical remarks}.
\invadraw{Embeddability (simplicial, PL or topological) of a complex in $\R^d$ is alternatively defined
as the existence of an injective (simplicial, PL or continuous) map of its body into $\R^d$.}}





\begin{proposition}\label{3-emble}
(a) Any connected 2-manifold that is either orientable or has non-empty boundary is PL embeddable into $\R^3$.

(b) Any 2-manifold is PL embeddable into $\R^4$.

(c) No closed non-orientable 2-manifold is PL embeddable into $\R^3$.
\end{proposition}

Already in the early history of topology mathematicians showed that in the General Position Theorem \ref{0-gp5}
the number $2k+1$ is the least possible.

\begin{proposition}\label{p:vkfm} For any $k$ any of the following $k$-complexes is
not PL (and hence simplicially) embeddable in $\R^{2k}$.

(a) The complete $k$-complex on $2k+3$ vertices, or the $k$-skeleton of the $(2k+2)$-simplex (this complex is $K_5$ for $k=1$; proved by Egbert van Kampen in 1932).

(b) The $k$-complex on $3(k+1)$ vertices split into $k+1$ triples, any $k+1$ vertices of distinct pairs forming a face;
this complex is $K_{3,3}$ for $k=1$ and is the $(k+1)$-join power $[3]^{*(k+1)}$ of the three-point set $[3]$ for arbitrary $k$; proved by A. Flores in 1934).

(c) The $k$-th power of a non-planar graph.
\end{proposition}


Short proof of the simplicial analogue of (a) is given in \cite[\S2, Proof of Theorem 1.4]{Sk14}; the proof generalizes to a proof of (a) itself.
Part (a) (and its TOP analogue\algor{, see \S\ref{s:topnr}}) follows from the van Kampen--Flores Theorem \ref{t:tvkf}.
Part (b) (and its TOP analogue) is deduced from the Borsuk-Ulam Theorem \algor{\ref{vf-bu} in \S\ref{s:topnr}.}
\invadraw{\cite[\S5]{Sk06}, \cite[\S5.7 `PL non-embeddability of $k$-complexes in $\R^{2k}$']{Sk}.}
Part (c) is conjectured by Karl Menger in 1929 but proved only by Brian Ummel in 1978 for $k=2$ \cite{Um78} and by Mikhail Skopenkov for an arbitrary $k$ in 2003 \cite{Sk03}, see exposition in \cite{Sk14}.
These are both early applications of {\it combinatorial topology} (nowadays called algebraic topology) and
the first results of {\it topological combinatorics} (also an area of ongoing active research).

\begin{theorem}[cf. Proposition \ref{grapl-ea}]\label{t:rec}
For every fixed $d,k$ such that either $k=2\ne d-2$ or $d\ge\frac{3k+3}2$ there is an algorithm for recognizing the PL embeddability of $k$-complexes in $\R^d$.
\end{theorem}

\algor{Theorem \ref{t:rec} for $k=d=2$ вытекает из аналогичного результата для графов --- теоремы \ref{grapl-ea}.c, или из аналога теоремы Фари \ref{grapl-fary} и утверждения \ref{t:recs}.
In \cite[Appendix A]{MTW} it is explained that Theorem \ref{t:rec} for $k=d=2$ (even with linear algorithm) follows из критерия типа Куратовского планарности 2-комплексов, т.е. из теоремы Халина-Юнга \ref{pla-hj}.b (но не из \ref{pla-hj}.a!).
Случай $d\ge3$ гораздо более сложен.}
\invadraw{In \cite[Appendix A]{MTW} it is explained that Theorem \ref{t:rec} for $k=d=2$ (even with linear algorithm) follows from the Kuratowski-type Halin-Jung planarity criterion for 2-complexes (stated there).}
Theorem \ref{t:rec} for $k=d-1=2$ is proved in \cite{MST+}.
In \cite[text after Theorem 1.4]{CKV}, \cite[\S1]{ST17} it is explained that Theorem \ref{t:rec} for $d\ge\frac{3k+3}2$ (even with polynomial algorithm) follows from \cite[Theorem 1.1]{CKV} and the Haefliger-Weber `configuration spaces' criterion \algor{\ref{dp-wu}} for embeddability of complexes\invadraw{ (stated there or in the survey \cite[Theorem 5.5]{Sk06})}.


The assumption of Theorem \ref{t:rec} is fulfilled when $d=2k\ge6$.
The idea of proof for $d=2k\ge6$ generalizes the proof of Proposition \ref{grapl-ea}.b presented in \S\ref{0vkam2}, see \algor{\S\ref{0vkam}.}
\invadraw{\cite[\S5.9 `Recognizing realizability of $k$-complexes in $\R^{2k}$']{Sk}.}

\begin{theorem}\label{t:undec1} For every fixed $d,k$ such that
$5\le d\in\{k,k+1\}$
there is no algorithm recognizing PL embeddability of $k$-complexes in $\R^d$.
\end{theorem}

This
is deduced in \cite[Theorem 1.1]{MTW} from the Novikov theorem on unrecognizability of the sphere.
The analogue of this for $8\le d\le\frac{3k+1}2$ is announced in the paper \cite{FWZ} containing a mistake \cite[\S3]{Sk20e} (see also \cite{KS20}).

\begin{theorem}\label{t:nphh} For every fixed $d,k$ such that $3\le d\le\frac{3k}2+1$ the algorithmic problem of recognizing PL embeddability of $k$-complexes in $\R^d$ is $NP$-hard.
\end{theorem}

This is proved for $d\ge4$ and $d=3$ in \cite{MTW} and in \cite{MRS+}, respectively.
See a simpler exposition  for $d\ge4$ in \cite{ST17} (where also a generalization was proved).
The proof for $d\ge4$ uses the construction \cite{SSS} of counterexamples to the Haefliger-Weber criterion \algor{\ref{dp-wu}} for embeddability of complexes.
For a `3- and 2-dimensional explanation' of ideas of proof see Propositions \ref{1-k5-1}, \ref{grapl-ram} and
\algor{п. \ref{0-nph}.}
\invadraw{\cite[\S5.11 `NP hardness of PL embeddability of complexes']{Sk}.}

\newcommand\unk{{\bf?}}
\newcommand\nee{~~~}
\newcommand\ja{{$+$}}
\newcommand\und{{UD}}
\newcommand\NPh{{NPh}}
\newcommand\NP{{NP}}
\newcommand\pol{{P}}

The following table summarizes the above results on the algorithmic problem
of recognizing PL embeddability of $k$-complexes in $\R^d$ (\ja\ $=$ always embeddable,
\pol\ $=$ polynomial-time solvable, D $=$ algorithmically decidable,
\NPh\ $=$ \NP-hard, \und\ $=$ algorithmically undecidable).

\begin{center}
\begin{tabular}{c|ccccccccccccc}
$k\backslash\ d$&~2~&~3~&4    &5    &6   &7   &8   &9   &10    &~11   &~12  &~13 &~14   \\
\hline
1 & \pol&\ja  &\ja  &\ja  &\ja &\ja &\ja &\ja &\ja   &\ja  &\ja &\ja & \ja   \\
2 & \pol&D,\NPh &\NPh &\ja  &\ja &\ja &\ja &\ja &\ja   &\ja  &\ja &\ja & \ja   \\
3 &\nee&D,\NPh &\NPh &\NPh &\pol &\ja &\ja &\ja &\ja   &\ja  &\ja &\ja & \ja   \\
4 &\nee&\nee &\NPh &\und &\NPh&\NPh&\pol &\ja &\ja   &\ja  &\ja &\ja & \ja   \\
5 &\nee&\nee &\nee &\und &\und&\NPh&\NPh&\pol&\pol   &\ja  &\ja &\ja & \ja   \\
6 &\nee&\nee &\nee &\nee &\und&\und&\NPh&\NPh&\NPh   &\pol &\pol&\ja & \ja   \\
7 &\nee&\nee &\nee &\nee &\nee&\und&\und&\NPh&\NPh   &\NPh &\pol&\pol& \pol
\end{tabular}
\end{center}

\subsection{Algorithmic recognition of almost realizablity of complexes}\label{s:hiae}

A (continuous, or PL) map  $f\colon K\to \R^d$ from a complex $K$ is an {\bf almost $r$-embedding} if $f\sigma_1\cap \ldots \cap f\sigma_r=\emptyset$ whenever $\sigma_1,\ldots,\sigma_r$ are pairwise disjoint faces of $K$.

\begin{remark}\label{r:ae} (a) In this language the Topological Tverberg Conjecture \ref{c:tt} and the $r$-fold van Kampen--Flores Conjecture \ref{c:vkfg} state that

($TT_{r,d}$) for every integers $r,d$ there are no almost $r$-embeddings $\Delta_{(d+1)(r-1)}\to\R^d$.

($VKF_{r,k}$) for every integers $r,k$ there are no almost $r$-embeddings of
the union of $k(r-1)$-faces of $\Delta_{(kr+2)(r-1)}$ in $\R^{kr}$.

We have $(TT_{r,kr+1})\Rightarrow(VKF_{r,k})$.
This was proved in \cite[2.9.c]{Gr10} and implicitly rediscovered in \cite[Lemma 4.1.iii and 4.2]{BFZ14}, \cite[proof of Theorem 4]{Fr15}; see survey \cite[Constraint Lemma 1.8 and Historical Remark 1.10]{Sk16}.

(b) The notion of an almost $2$-embedding implicitly appeared in studies realizability of graphs and complexes (Theorems \ref{grapl-nonalm}, \ref{t:tvkf} and \ref{t:rec}).
It was explicitly formulated in the Freedman-Krushkal-Teichner work on the van Kampen obstruction \cite{FKT}.

(c) Any sufficiently small perturbation of an almost $r$-embedding is again an almost $r$-embedding.
So the existence of a {\it continuous} almost $r$-embedding is equivalent to the existence of a {\it PL} almost $r$-embedding, and to the existence of a {\it general position PL} almost $r$-embedding.
Cf. \cite[Approximation Lemma 1.4.6]{Sk20}.
\end{remark}

See more introduction in \cite[\S1.2]{Sk16}.

\begin{problem}\label{p:2ae2} Which 2-complexes admit a PL map to $\R^2$ without triple points?
Which 2-complexes are almost 3-embeddable in $\R^2$?
Are there algorithms for checking the above properties of 2-complexes?
Same questions for $\R^2$ replaced by $\R^3$, or for `triple' and `almost 3-embeddable' replaced by `$r$-tuple' and `almost $r$-embeddable'.\footnote{Analogous problems for maps from graphs to the line are investigated in studies of {\it cutwidth}, see \cite{TSB, LY04, Kho} and references therein.}
\end{problem}



\begin{theorem}\label{t:ozmawaco} If $r$ is not a prime power, then

(a) for any $k\ge2$ there is an almost $r$-embedding of any  $k(r-1)$-complex in $\R^{kr}$.
\cite{MW15, AMSW}

(b) there is an almost $r$-embedding of any $s$-complex in $\R^{s+\big\lceil\tfrac{s+3}r\big\rceil}$.
\cite{AKS}
\end{theorem}

Part (a) follows from Theorems \ref{t:mawa} and \ref{t:ozmawa} below.

\begin{theorem}[\cite{MW16, AMSW, Sk17, Sk17o}]\label{t:algal}
For every fixed $k,d,r$ such that either $rd\ge(r+1)k+3$ or $d=2r=k+2\ne4$ there is a polynomial algorithm for checking PL almost $r$-embeddability of $k$-complexes in $\R^d$.
\end{theorem}

For $(r-1)d=rk$ Theorem \ref{t:algal} was deduced in \cite{MW15, AMSW} from Theorem \ref{t:mawa} below.


For a version of Theorem \ref{t:nphh} with `embeddability' replaced by `almost 2-embeddability' see  \cite{ST17}.


We shall state the \"Ozaydin theorem \cite{Oz} (generalizing the \"Ozaydin Theorem in the plane \ref{vankamz-oz}) in the simplified form of Theorem \ref{t:ozmawa} below.
This is different from the standard form \cite[Theorem 3.3]{Sk16} but is equivalent to the standard form by a proposition of Mabillard-Wagner \cite[Proposition 3.4]{Sk16}.
For the statement we need the following definitions.

\begin{figure}[h]
\centerline{\includegraphics[width=8cm]{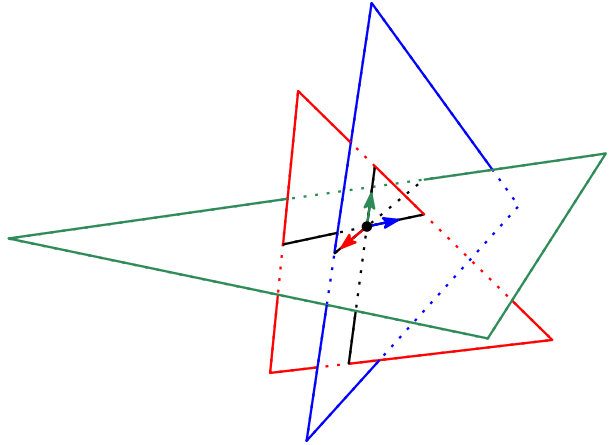}}
\caption{A 3-fold point and its 3-intersection sign}\label{f:gl3}
\end{figure}

Let $K$ be a $k(r-1)$-complex for some $k\ge1$, $r\ge2$, and $f\colon K\to \R^{kr}$ a PL map in general position.

Then preimages $y_1,\ldots,y_r\in K$ of any $r$-fold point $y\in\R^{kr}$ (i.e., of a point having
$r$ preimages) lie in the interiors of $k(r-1)$-dimensional simplices of $K$.
Choose arbitrarily an orientation for each of the $k(r-1)$-simplices.
By general position, $f$ is affine on a neighborhood $U_j$ of $y_j$ for each $j=1,\ldots,r$.
Take a positive basis of $k$ vectors in the oriented normal space to oriented $fU_j$.
{\it The $r$-intersection sign} of $y$ is the sign $\pm 1$ of the basis in $\R^{kr}$ formed by $r$ such $k$-bases.
See figs. \ref{f:sign} and \ref{f:gl3}.

This is classical for $r=2$, see \S\ref{0thint}, and is analogous for $r\ge3$, cf. \S\ref{s:triple}, \cite[\S2.2]{MW15}.

We call the map $f$ a {\bf $\Z$-almost $r$-embedding} if $f\sigma_1\cdot\ldots\cdot f\sigma_r=0$ whenever $\sigma_1,\ldots,\sigma_r$ are pairwise disjoint simplices of $K$.
Here the {\it algebraic $r$-intersection number}  $f\sigma_1\cdot\ldots\cdot f\sigma_r\in\Z$ is defined as the sum of the $r$-intersection signs of all $r$-fold points $y\in f\sigma_1\cap\ldots\cap f\sigma_r$.
The sign of $f\sigma_1\cdot\ldots\cdot f\sigma_r$ depends on an arbitrary choice of orientations for each $\sigma_i$ and on the order of $\sigma_1,\ldots,\sigma_r$, but the condition $f\sigma_1\cdot\ldots\cdot f\sigma_r=0$ does not.
See fig. \ref{f:gl2} for $r=2$.

Clearly, an almost $r$-embedding is a $\Z$-almost $r$-embedding.

\begin{theorem}[\cite{MW15, AMSW}; cf. Theorems \ref{vkam-z2} and \ref{vkam-z2-r}]\label{t:mawa}
If $k\ge2$, $k+r\ge5$ and there is a $\Z$-almost $r$-embedding of a $k(r-1)$-complex $K$ in $\R^{kr}$, then there is an almost $r$-embedding of $K$ in $\R^{kr}$.
\end{theorem}

\begin{theorem}[cf. {\cite[Theorem 4]{AKS}, \cite[Theorem 5.1]{AK19}}]\label{t:ozmawa} If $r$ is not a prime power and $k\ge2$, then there is a $\Z$-almost $r$-embedding of any  $k(r-1)$-complex in $\R^{kr}$.
\end{theorem}


{\bf Books, surveys and expository papers in this list are marked by the stars.}

\end{document}